\documentclass[12pt,a4paper]{amsart}

\usepackage[english]{babel}
\usepackage[utf8x]{inputenc}
\usepackage{amsfonts}
\usepackage{amssymb}
\usepackage{amsmath}
\usepackage{amsrefs}
\usepackage{mathrsfs}
\usepackage{graphicx}
\usepackage{framed}
\usepackage{mdframed}
\usepackage[colorinlistoftodos]{todonotes}
\usepackage{color}

\usepackage{dsfont}
\usepackage{stmaryrd}

\usepackage[scr=boondox]{mathalfa}

\newtheorem{theorem}{Theorem}[section]
\newtheorem{lemma}[theorem]{Lemma}
\newtheorem{proposition}[theorem]{Proposition}

\newtheorem{corollary}[theorem]{Corollary}

\theoremstyle{definition}
\newtheorem{definition}[theorem]{Definition}

\newtheorem*{ack}{Acknowledgements}

\newcommand{\ad}{\mathrm{ad}}
\newcommand{\Ad}{\mathrm{Ad}}
\DeclareMathOperator{\Hol}{\mathrm{Hol}}
\DeclareMathOperator{\Aut}{\mathrm{Aut}}
\DeclareMathOperator{\Iso}{\mathrm{I}}
\DeclareMathOperator{\Orth}{\mathrm{O}}

\DeclareMathOperator{\GLin}{\mathrm{GL}}

\newcommand{\MC}[1]{\omega_{{}_{#1}}}
\newcommand{\Lt}[1]{\operatorname{L}_{#1}}
\newcommand{\Rt}[1]{\operatorname{R}_{#1}}

\DeclareFontEncoding{T3}{}{}
\DeclareFontSubstitution{T3}{cmr}{m}{n}
\DeclareSymbolFont{tipa}{T3}{cmr}{m}{sl}
\SetSymbolFont{tipa}{bold}{T3}{cmr}{bx}{sl}

\DeclareMathSymbol{\kgf}{\mathord}{tipa}{'255}

\usepackage[all]{xy}

\usepackage[normalem]{ulem}

\title[Intrinsic holonomy and curved cosets]{Intrinsic holonomy and curved cosets of Cartan geometries}
\author{Jacob W. Erickson}
\date{\today}

\begin{document}
\maketitle

\begin{abstract}We provide an intrinsic notion of curved cosets for arbitrary Cartan geometries, simplifying the existing construction of curved orbits for a given holonomy reduction. To do this, we define an intrinsic holonomy group, which is shown to coincide precisely with the standard definition of the holonomy group for Cartan geometries in terms of an associated principal connection. These curved cosets retain many characteristics of their homogeneous counterparts, and they behave well under the action of automorphisms. We conclude the paper by using the machinery developed to generalize the de Rham decomposition theorem for Riemannian manifolds and give a potentially useful characterization of inessential automorphism groups for parabolic geometries.\end{abstract}


\section{Introduction}
In motivating the study of Cartan geometries, we are often told that Cartan geometries resemble Lie groups in a particular way, allowing us to extend the spirit of Klein's Erlangen program to geometric structures like Riemannian manifolds that are not necessarily homogeneous. From this, an astonishingly extensive portion of the vast menagerie of geometric structures studied in differential geometry can be brought under a single unified perspective.

The recent renaissance in the study of parabolic geometries over the past few decades bears testament to the enormous benefit of this frame of mind. For example, using this resemblance of Cartan geometries to Lie groups, sequences of natural differential operators analogous to the (dual) Bernstein-Gelfand-Gelfand resolutions in representation theory were constructed in \cite{CSV} (later simplified by \cite{CalderbankDiemer}) for parabolic geometries of arbitrary type.

And yet, in the long shadow cast by these towering achievements of differential geometry by way of representation theory, several fundamental questions still seem astoundingly unexplored. For example, if a Cartan geometry is like a Lie group, then what do its subgroups look like? Would the notion of subgroup even be useful for Cartan geometries, let alone make sense? If so, then how much like actual subgroups can we expect them to behave?

As we will see shortly, the answers to these questions come, at least in part, from holonomy reductions. These were explored for Cartan geometries in \cite{CGH}, but while the constructions in that paper were fairly straightforward to implement in most applications, the approach was extrinsic, relying on an extension of the Cartan connection to a principal connection on a larger bundle, and the role of the analogy between Cartan geometries and Lie groups was, in our opinion, somewhat obscured by this. Thus, the first part of this paper will be devoted to providing an intrinsic reframing of the machinery of \cite{CGH}. To do this, we will give an intrinsic definition of holonomy, then use this intrinsic definition to specify ``subgroups'' and ``cosets'' of the Cartan geometry.

Excitingly, our notion of intrinsic holonomy allows us to make the analogy between Cartan geometries and Lie groups quite explicit, and to highlight this, we will attempt to fully commit to the pretense of doing ``Lie theory" on Cartan geometries. Nearly all results in the main body of the paper are directly motivated by analogous results for Lie groups, and in some cases, proofs will deliberately mirror well-known proofs from elementary Lie theory. We have found that this ``pretend Lie theory" perspective is quite appealing to those not already familiar with Cartan geometries, and hope that presenting our ideas in this somewhat avant-garde way both clarifies their role in the subject and encourages others to explore this area of geometry.

We have organized the paper as follows. After giving some preliminary definitions in Section \ref{preliminaries}, we start in Section \ref{intrinsicholonomy} by giving a more natural definition of holonomy, and then show that this intrinsic holonomy is actually equivalent to the extrinsic holonomy used in \cite{CGH}. In Section \ref{curvedcosets}, we introduce the notion of curved cosets and demonstrate some ways in which they behave like genuine cosets of Lie groups. Section \ref{tractorconnections} deals with tractor bundles and tractor connections, and while none of the results in that section are particularly new, we think that the perspective used there is aesthetically pleasant and connects well with the machinery developed in Section \ref{curvedcosets}. The way that automorphisms of Cartan geometries interact with holonomy and curved cosets is explored in Section \ref{automorphisms}. Finally, to exhibit possible applications beyond those described in \cite{CGH}, we end the paper in Section \ref{applications} with a generalization of the de Rham decomposition theorem for Riemannian manifolds and a new situationally useful approach to essential automorphisms of parabolic geometries.

\begin{ack}The author would like to thank Andreas \v{C}ap for his helpful commentary on several drafts of this paper, in particular for his critique of earlier drafts.\end{ack}

\section{Preliminaries}\label{preliminaries}
We begin by recalling some basic definitions from the theory of Cartan geometries that will be used throughout the paper. The definitions of tractor connections, automorphisms, extension functors, and completeness, while standard aspects of the subject, have been deferred to later sections for clarity of exposition. For a more detailed introduction, the author recommends \cite{CapSlovak}. Alternatively, an intuitive but slightly less scrupulous introduction can be found in \cite{Sharpe}.

Each Cartan geometry is modeled on a homogeneous geometric structure in the sense of Klein.

\begin{definition}A \emph{model} or \emph{model geometry} is a pair $(G,H)$, where $G$ is a Lie group and $H$ is a closed subgroup of $G$ such that $G/H$ is connected.\end{definition}

If $(G,H)$ is a model, then $G$ is a principal $H$-bundle over $G/H$. The geometric information of the model is encoded by the Maurer-Cartan form $\MC{G}$ on $G$, which takes each $X_g\in T_gG$ to $\Lt{g^{-1}*}X_g\in\mathfrak{g}$. To extend this perspective, we replace $G$ with a principal $H$-bundle over a smooth manifold $M$ and $\MC{G}$ with a $\mathfrak{g}$-valued $1$-form called a \emph{Cartan connection} that satisfies a few similar properties.

\begin{definition}Suppose $M$ is a smooth manifold and $(G,H)$ is a model. Then, a \emph{Cartan geometry of type $(G,H)$} over $M$ is a pair $(\mathscr{G},\omega)$, where $q_{{}_H}\hspace{-0.4em}:\mathscr{G}\to M$ is a principal $H$-bundle over $M$ and $\omega$ is a $\mathfrak{g}$-valued $1$-form on $\mathscr{G}$ such that the following are satisfied:
\begin{itemize}
\item For every $\mathscr{g}\in\mathscr{G}$, $\omega_\mathscr{g}:T_\mathscr{g}\mathscr{G}\to\mathfrak{g}$ is a linear isomorphism;
\item For every $h\in H$, $\Rt{h}^*\omega=\Ad_{h^{-1}}\omega$\,;
\item For every $Y\in\mathfrak{h}$ and $\mathscr{g}\in\mathscr{G}$, the flow of the vector field $\omega^{-1}(Y)$ starting at $\mathscr{g}$ is given by $\exp(t\omega^{-1}(Y))\mathscr{g}=\mathscr{g}\exp(tY)$.\end{itemize}\end{definition}

This description of geometric structure encompasses many of the various flavors of modern differential geometry, putting them all under a common framework. As a motivating example, an $m$-dimensional Riemannian manifold $(M,\mathrm{g})$ determines a unique (torsion-free) Cartan geometry of type $(\Iso(m),\Orth(m))$ over $M$, where $\Iso(m)$ is the isometry group of $m$-dimensional Euclidean space, and a (torsion-free) Cartan geometry of type $(\Iso(m),\Orth(m))$ over $M$ determines a Riemannian metric on $M$ that is unique up to scaling.

Given a model $(G,H)$, a helpful example of a Cartan geometry of type $(G,H)$ to keep in mind is the corresponding \emph{Klein geometry} $(G,\MC{G})$ of type $(G,H)$ over $G/H$, which gives the model as a Cartan geometry. When a Cartan geometry is not equivalent to the Klein geometry, its curvature indicates how they differ locally.

\begin{definition}Let $(\mathscr{G},\omega)$ be a Cartan geometry of type $(G,H)$. The \emph{curvature} $\Omega$ of $(\mathscr{G},\omega)$ is given by $\Omega=\mathrm{d}\omega+\frac{1}{2}[\omega,\omega]$. In particular, for $X,Y\in\mathfrak{g}$, the curvature applied to $\omega^{-1}(X)\wedge\omega^{-1}(Y)$ is given by \begin{align*}\Omega(\omega^{-1}(X)\wedge\omega^{-1}(Y)) & =\mathrm{d}\omega(\omega^{-1}(X)\wedge\omega^{-1}(Y))+[X,Y] \\ & =\omega^{-1}(X)\omega(\omega^{-1}(Y))-\omega^{-1}(Y)\omega(\omega^{-1}(X)) \\ & \hspace{5em} -\omega([\omega^{-1}(X),\omega^{-1}(Y)])+[X,Y] \\ & =[X,Y]-\omega([\omega^{-1}(X),\omega^{-1}(Y)]).\end{align*} We will frequently use the notation $\Omega^{\omega}(X\wedge Y):=\Omega(\omega^{-1}(X)\wedge\omega^{-1}(Y))$. When $\Omega$ vanishes identically, we say the Cartan geometry is \emph{flat}.\end{definition}

When the curvature vanishes everywhere, the Cartan geometry is locally equivalent to the model geometry. In other words, a flat Cartan geometry is equivalent to a locally homogeneous geometric structure along the lines of those used in Thurston's geometrization conjecture.

One particularly important tool for Cartan geometries is the idea of development, which allows us to compare paths within a geometry with corresponding paths within the model.

\begin{definition}Suppose $\gamma:[0,1]\to\mathscr{G}$ is a path in a Cartan geometry $(\mathscr{G},\omega)$ of type $(G,H)$. Then, the \emph{development} of $\gamma$ on $G$ is the unique path $\gamma_{{}_G}:[0,1]\to G$ starting at the identity in $G$ such that $\gamma^*\omega=\gamma_{{}_G}^*\MC{G}$.\end{definition}

\begin{figure}
\centering\includegraphics[width=0.8\textwidth]{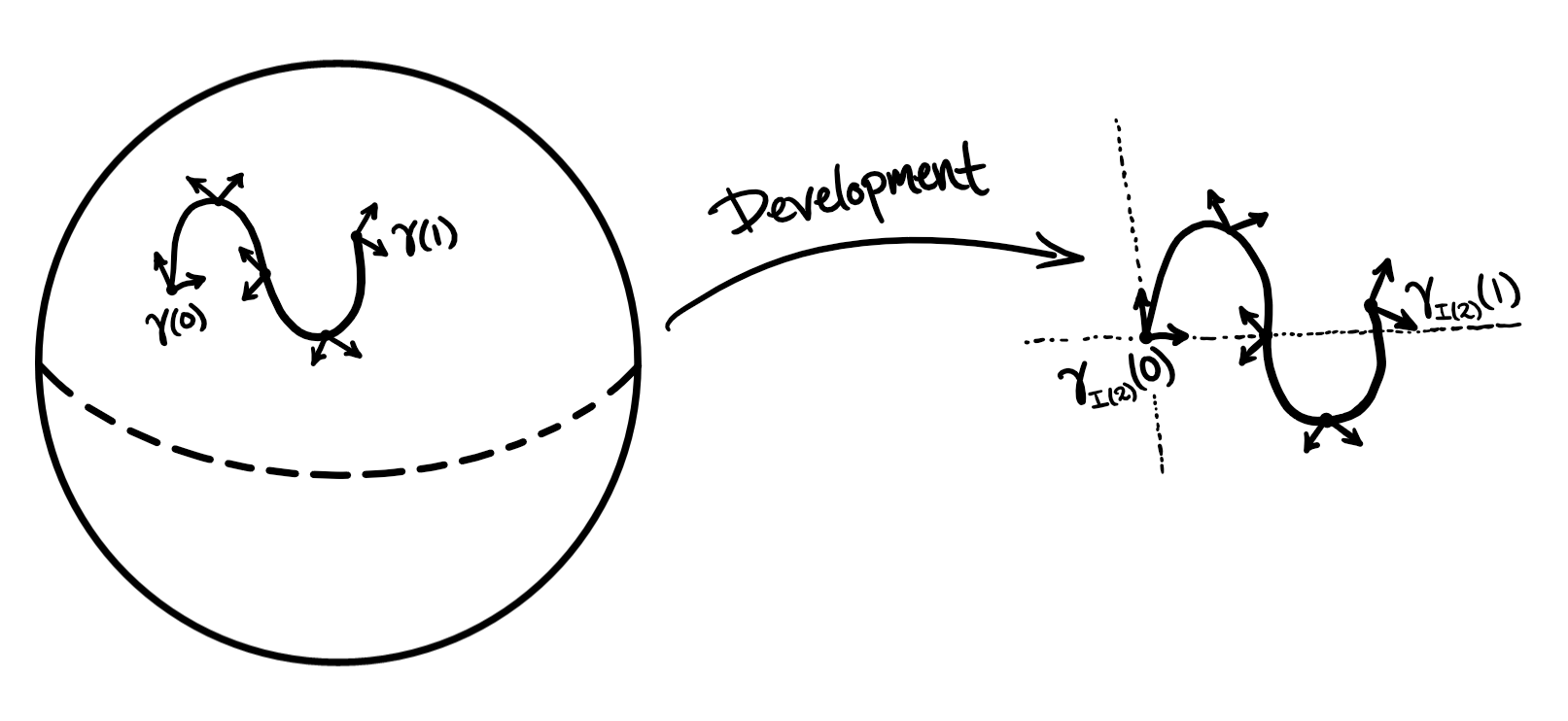}
\caption{A drawing of a path in the orthonormal frame bundle of the sphere equipped with a Cartan connection of type $(\Iso(2),\Orth(2))$, and its development in $\Iso(2)$ (thought of as the orthonormal frame bundle over the Euclidean plane)}
\label{devpic}
\end{figure}

Intuitively, $\gamma$ gives us instructions for a path on $\mathscr{G}$ in terms of our Cartan connection, and $\gamma_{{}_G}$ is the path we take if we try to follow these same instructions using the Maurer-Cartan form on $G$, starting at the identity. For example, if $\gamma$ is already a path on $G$, then $\gamma_{{}_G}$ is just the left-translate $\Lt{\gamma(0)^{-1}}\circ\,\gamma$.\vspace{.5em}

We now turn to the machinery described in \cite{CGH}. Given a Cartan geometry $(\mathscr{G},\omega)$ of type $(G,H)$ over $M$, we can construct a principal $G$-bundle $\widehat{\mathscr{G}}=\mathscr{G}\times_H G$ and a principal connection $\widehat{\omega}$ on $\widehat{\mathscr{G}}$ that is characterized by $j^*\widehat{\omega}=\omega$, where $j:\mathscr{G}\to\widehat{\mathscr{G}},~\mathscr{g}\mapsto(\mathscr{g},e)$. This allows us to define an extrinsic notion of holonomy.

\begin{definition}The \emph{extrinsic holonomy group} of $(\mathscr{G},\omega)$ at $\widehat{u}\in\widehat{\mathscr{G}}$ is the holonomy group $\Hol_{\widehat{u}}(\widehat{\mathscr{G}},\widehat{\omega})$ of the principal connection $\widehat{\omega}$ at $\widehat{u}$.\end{definition}

In \cite{BaumJuhl} and \cite{Sharpe}, a more natural definition of holonomy was given, where the holonomy group was defined to be the subgroup of $G$ whose elements were developments along loops within the Cartan geometry. While this idea is on the right track, it has several issues that make it less viable than the extrinsic holonomy given above. For example, this alternative definition does not capture the orientation-reversing part of the holonomy of the locally Euclidean Klein bottle, since the corresponding loops on the Klein bottle do not lift to loops in the orthonormal frame bundle. In the next section, we will give a corrected version of this idea.

With holonomy in hand, we can define the notion of holonomy reduction used in \cite{CGH}.

\begin{definition}Suppose $(\mathscr{G},\omega)$ is a Cartan geometry of type $(G,H)$ over $M$ and $X$ is a homogeneous space for $G$. Then, a \emph{holonomy reduction of $G$-type $X$} for $(\mathscr{G},\omega)$ is defined to be a section $\sigma:M\to\widehat{\mathscr{G}}\times_GX$ that is parallel with respect to the connection induced by $\widehat{\omega}$.\end{definition}

This is called a holonomy reduction because, for $X=G/J$, it is equivalent to a reduction of structure group on $\widehat{\mathscr{G}}$ from $G$ to $J$ such that $\widehat{\omega}$ pulls back to a principal $J$-connection under the corresponding inclusion, and the smallest choice of $J$ is the holonomy group. In other words, it is equivalent to a holonomy reduction of the extended principal bundle with the extended principal connection.

Of course, a section $\sigma$ of $\widehat{\mathscr{G}}\times_GX$ is equivalent to a $G$-equivariant map $\sigma:\widehat{\mathscr{G}}\to X$, and both are equivalent to an $H$-equivariant map $\sigma:\mathscr{G}\to X$; we will follow the practice of not distinguishing between these maps.

We can now define curved orbits as collections of points whose fibers in $\mathscr{G}$ have the same image under a holonomy reduction $\sigma$.

\begin{definition}Suppose $(\mathscr{G},\omega)$ is a Cartan geometry of type $(G,H)$ over $M$ together with a holonomy reduction $\sigma:\mathscr{G}\to X$. The \emph{$H$-type} of a point $p\in M$ is the image $\sigma(q_{{}_H}^{-1}(p))\subseteq X$ of the fiber of $q_{{}_H}:\mathscr{G}\to M$ over $p$, and the collection of all $p\in M$ with the same $H$-type is called a \emph{curved orbit}.\end{definition}

From this definition, it is not immediately apparent whether a curved orbit is anything more than a subset of $M$. Nevertheless, it was proved in \cite{CGH} that each curved orbit is an initial submanifold of $M$, and that each of them carries a canonical geometric structure. Later, we will give an alternative perspective that makes this inherent structure obvious.

\section{Intrinsic holonomy}\label{intrinsicholonomy}
We begin with a corrected definition of holonomy.

\begin{definition}Suppose $(\mathscr{G},\omega)$ is a Cartan geometry of type $(G,H)$ over $M$. The \emph{(intrinsic) holonomy} with respect to $\mathscr{g}\in\mathscr{G}$ of a loop $\gamma:[0,1]\to M$ such that $\gamma(0)=\gamma(1)=q_{{}_H}(\mathscr{g})$ is the element $\widehat{\gamma}_{{}_G}(1)h_{\widehat{\gamma}}\in G$, where $\widehat{\gamma}$ is a lift of $\gamma$ to $\mathscr{G}$ starting at $\widehat{\gamma}(0)=\mathscr{g}$ and $h_{\widehat{\gamma}}\in H$ is such that $\widehat{\gamma}(1)h_{\widehat{\gamma}}=\mathscr{g}=\widehat{\gamma}(0)$. The \emph{(intrinsic) holonomy group} $\Hol_\mathscr{g}(\mathscr{G},\omega)$ at $\mathscr{g}$ is the subgroup of $G$ consisting of all such holonomies. In other words, \[\Hol_\mathscr{g}(\mathscr{G},\omega):=\left\{\gamma_{{}_G}(1)h\in G: \begin{matrix}h\in H\text{ and }\gamma:[0,1]\to\mathscr{G}\\ \text{ such that }\gamma(0)=\mathscr{g}=\gamma(1)h\end{matrix}\right\}.\]\end{definition}

Note that the holonomy of a loop $\gamma$ in the base manifold only depends on that loop and the chosen basepoint $\mathscr{g}\in\mathscr{G}$, since if $\widehat{\gamma}$ is our chosen lift starting at $\mathscr{g}$, then every other lift starting at $\mathscr{g}$ will be of the form $\Rt{\beta}(\widehat{\gamma})$ for some path $\beta$ in $H$ starting at the identity, and $\Rt{\beta}(\widehat{\gamma})_{{}_G}=\Rt{\beta}(\widehat{\gamma}_{{}_G})$. Furthermore, this definition remains natural while still fixing the issues in the definitions given by \cite{BaumJuhl} and \cite{Sharpe}. Moreover, it is actually equivalent to the extrinsic holonomy.

\begin{proposition}For every Cartan geometry $(\mathscr{G},\omega)$ of type $(G,H)$, \[\Hol_\mathscr{g}(\mathscr{G},\omega)=\Hol_{(\mathscr{g},e)}(\widehat{\mathscr{G}},\widehat{\omega}).\]\end{proposition}
\begin{proof}The horizontal curves of $(\widehat{\mathscr{G}},\widehat{\omega})$ are precisely those curves that take the form $t\mapsto (\gamma(t),(\gamma_{{}_G}(t))^{-1}g')$. Thus, if $\gamma:[0,1]\to\mathscr{G}$ is a lift of a loop in $M$ and $h\in H$ is such that $\gamma(0)=\gamma(1)h$, then the horizontal path $t\mapsto(\gamma(t),(\gamma_{{}_G}(t))^{-1})$ starts at $(\gamma(0),e)$ and ends at \[(\gamma(1),(\gamma_{{}_G}(1))^{-1})=(\gamma(1)h,h^{-1}(\gamma_{{}_G}(1))^{-1})=(\gamma(0),(\gamma_{{}_G}(1)h)^{-1}).\] Since $\Hol_{(\mathscr{g},e)}(\widehat{\mathscr{G}},\widehat{\omega})$ is a subgroup, $(\Hol_{(\mathscr{g},e)}(\widehat{\mathscr{G}},\widehat{\omega}))^{-1}=\Hol_{(\mathscr{g},e)}(\widehat{\mathscr{G}},\widehat{\omega})$, so the extrinsic holonomy group is equal to the intrinsic holonomy group.\mbox{\qedhere}\end{proof}

The straightforwardness of this proof makes it particularly difficult for the author to believe that our small adjustment to the definition of holonomy used in \cite{BaumJuhl} and \cite{Sharpe} was not present in any of the previous literature that he could find. Nonetheless, while there are results in the literature\footnote{See, for example, Proposition 2.1.1 of \cite{BaumJuhl}, which notes that the subgroups corresponding to holonomy from contractible loops coincide.} that relate the extrinsic holonomy group to the holonomy group consisting of developments of closed loops in $\mathscr{G}$ that lift from loops in $M$, our definition does not seem to have been considered before in the form presented above.

\begin{figure}
\centering\includegraphics[width=0.5\textwidth]{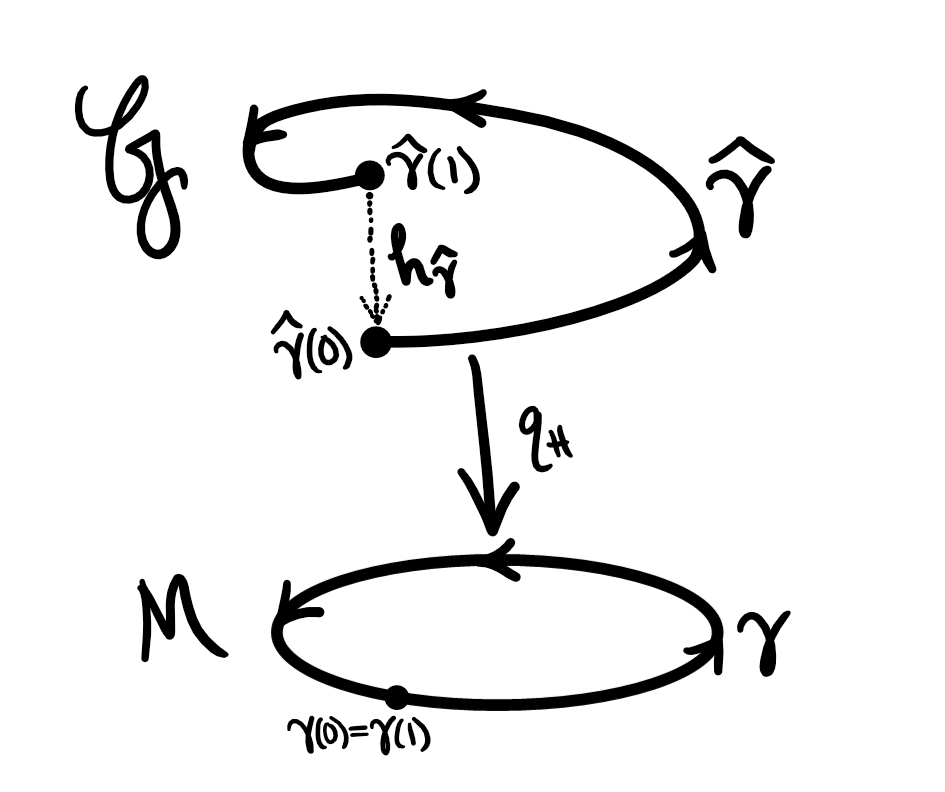}
\caption{Even if a loop $\gamma$ in $M$ does not lift to a loop in $\mathscr{G}$, we can still lift to a path $\widehat{\gamma}$ and then ``close the loop'' by right translating by an element of $H$. Developing along $\widehat{\gamma}$ and then right translating by that element of $H$ gives the intrinsic holonomy of $\gamma$ with respect to $\widehat{\gamma}(0)$.}
\end{figure}

In a similar way, we can also simplify what we mean by a holonomy reduction.

\begin{proposition}\label{basepoint} Suppose $M$ is connected. Then, a choice of basepoint $\mathscr{e}\in\mathscr{G}$ together with a closed subgroup $J\leq G$ containing the holonomy group $\Hol_\mathscr{e}(\mathscr{G},\omega)$ uniquely determines a holonomy reduction $\sigma:\mathscr{G}\to G/J$. Conversely, a holonomy reduction $\sigma:\mathscr{G}\to G/J$ determines a set of possible basepoints in $\mathscr{G}$.\end{proposition}
\begin{proof}Parallel sections of $\widehat{\mathscr{G}}\times_G(G/J)$ are precisely those which are constant on the horizontal curves. In other words, all holonomy reductions are of the form \[\sigma:(\gamma(1)h,g')\mapsto(g')^{-1}(\gamma_{{}_G}(1)h)^{-1}\sigma(\gamma(0),e).\] Thus, if we pick a point $\mathscr{e}\in\mathscr{G}$ and a closed subgroup $J\geq\Hol_\mathscr{e}(\mathscr{G},\omega)$, and we set $\sigma(\mathscr{e})=eJ$, then $\sigma$ is simply the map $\sigma:\gamma(1)h\mapsto(\gamma_{{}_G}(1)h)^{-1}J$, where $\gamma(0)=\mathscr{e}$. Note that this is a well-defined map, since every element of $\mathscr{G}$ can be written as $\gamma(1)h$ for a path $\gamma:[0,1]\to\mathscr{G}$ starting at $\mathscr{e}$, and if $\gamma_1(1)h_1=\gamma_2(1)h_2$, then $\gamma_{1{}_G}(1)h_1(\gamma_{2{}_G}(1)h_2)^{-1}\in\Hol_\mathscr{e}(\mathscr{G},\omega)\leq J$, so \[\sigma(\gamma_1(1)h_1)=(\gamma_{1{}_G}(1)h_1)^{-1}J=(\gamma_{2{}_G}(1)h_2)^{-1}J=\sigma(\gamma_2(1)h_2).\]

Conversely, if $\sigma:\mathscr{G}\to G/J$ is a holonomy reduction, then $\sigma^{-1}(eJ)$ gives the set of all basepoints which characterize $\sigma$ in the way described above.\mbox{\qedhere}\end{proof}

Note that, if $\sigma^{-1}(eJ)$ is empty, then we can use a $g\in G$ such that $\sigma^{-1}(g^{-1}J)$ is not empty; the corresponding closed subgroup will change from $J$ to $g^{-1}Jg$.

It is well-known in the principal bundle context that a choice of closed subgroup $J$ containing the holonomy group determines a holonomy reduction, so the existence of a result like Proposition \ref{basepoint} is not particularly surprising. However, viewed through the heuristic analogy where we pretend $(\mathscr{G},\omega)$ is a Lie group, the idea that a holonomy reduction just amounts to specifying a basepoint and a particular closed subgroup takes on additional meaning. Indeed, if we pretend that $\mathscr{e}$ is the ``identity element'' of $(\mathscr{G},\omega)$, then this allows us to specify a ``subgroup'' $\sigma^{-1}(eJ)$ of $(\mathscr{G},\omega)$, corresponding to $J\leq G$. In the next section, we will see that the ``subgroup'' $\sigma^{-1}(eJ)$ is an example of what we call a \emph{curved coset}.

\section{Curved cosets}\label{curvedcosets}
In this section, we will define curved cosets, describe some of their properties, and show how they are related to the curved orbits of \cite{CGH}. To start, however, we need a way to compare elements of a Cartan geometry with elements in its corresponding model.

\begin{definition}Suppose $(\mathscr{G},\omega)$ is a Cartan geometry of type $(G,H)$ and $\mathscr{e},\mathscr{g}\in\mathscr{G}$. We say that $g\in G$ is a \emph{development of $\mathscr{g}$ from $\mathscr{e}$} whenever there exists a path $\gamma:[0,1]\to\mathscr{G}$ and $h\in H$ such that $\gamma(0)=\mathscr{e}$, $\gamma(1)h=\mathscr{g}$, and $\gamma_{{}_G}(1)h=g$. When our choice of $\mathscr{e}$ is understood, we will just say that $g$ is a development of $\mathscr{g}$.\end{definition}

Unlike developments of paths, developments of points of $\mathscr{G}$ are not necessarily unique. Indeed, suppose $\alpha:[0,1]\to\mathscr{G}$ is a lift of a loop in $M$, so that $\alpha(0)=\mathscr{e}$ and $\alpha(1)=\mathscr{e}k$ for some $k\in H$. Then, if $\gamma:[0,1]\to\mathscr{G}$ is a path with $\gamma(0)=\mathscr{e}$ and $\gamma(1)h=\mathscr{g}$ for some $h\in H$, then concatenating $\alpha$ with the right-translate $\Rt{k}(\gamma)$ gives a path from $\mathscr{e}$ to $\mathscr{g}h^{-1}k$, so both $\gamma_{{}_G}(1)h$ and $\alpha_{{}_G}(1)(k^{-1}\gamma_{{}_G}(1)k)(h^{-1}k)^{-1}=(\alpha_{{}_G}(1)k^{-1})(\gamma_{{}_G}(1)h)$ are developments of $\mathscr{g}$ from $\mathscr{e}$. More succinctly, if $g$ is a development of $\mathscr{g}$ from $\mathscr{e}$, then so is every element of $\Hol_\mathscr{e}(\mathscr{G},\omega)g$.

Thus, in the heuristic analogy between $(\mathscr{G},\omega)$ and $G$, if we think of $\mathscr{e}$ as the ``identity element", then elements of $\mathscr{G}$ correspond to elements of $G$, the elements of $G$ to which they correspond are given by their developments from $\mathscr{e}$, and the holonomy group $\Hol_\mathscr{e}(\mathscr{G},\omega)$ tells us exactly if and how this correspondence fails to assign a unique group element.

In the Klein geometry $(G,\MC{G})$ of type $(G,H)$, where the holonomy group is always trivial, the unique development of $g_1\in G$ from $g_0\in G$ is $g_0^{-1}g_1$, so the (left) coset $g_0J$ is precisely the set of all $g_1$ with development from $g_0$ contained in $J$. Analogously, we define \emph{curved (left) cosets} of a Cartan geometry to be sets of points with developments contained in a given subgroup of $G$.

\begin{definition}\label{leftcoset} Suppose $(\mathscr{G},\omega)$ is a Cartan geometry of type $(G,H)$. Let $\mathscr{g}\in\mathscr{G}$ and let $J\leq G$ be a subgroup containing $\Hol_\mathscr{g}(\mathscr{G},\omega)$. Then, the \emph{curved (left) coset} $\mathscr{g}\bullet J$ through $\mathscr{g}$ is given by the set of all points in $\mathscr{G}$ with development from $\mathscr{g}$ contained in $J$. In other words, \[\mathscr{g}\bullet J:=\left\{\gamma(1)h:~\begin{matrix}h\in H\text{ and }\gamma:[0,1]\to\mathscr{G}\text{ such} \\ \text{that }\gamma(0)=\mathscr{g}\text{ and }\gamma_{{}_G}(1)h\in J\end{matrix}\right\}.\]\end{definition}

While the above definition works without the assumption that $J$ contains $\Hol_\mathscr{g}(\mathscr{G},\omega)$, $\Hol_\mathscr{g}(\mathscr{G},\omega)$ is the smallest sensible candidate for a choice of $J$, since by the considerations above, $\mathscr{g}\bullet\{e\}=\mathscr{g}\bullet\Hol_\mathscr{g}(\mathscr{G},\omega)$. Therefore, unless we indicate otherwise, we will assume that curved (left) cosets $\mathscr{g}\bullet J$ are \emph{well-defined}, meaning that $\Hol_\mathscr{g}(\mathscr{G},\omega)\leq J$.

For a given $\mathscr{g}\in\mathscr{G}$, $\mathscr{g}\bullet J$ being well-defined does not necessarily imply that $\mathscr{g'}\bullet J$ is well-defined for $\mathscr{g}'\not\in\mathscr{g}\bullet J$. For $\gamma:[0,1]\to\mathscr{G}$ a path starting at $\mathscr{g}$ such that $\gamma(1)h'=\mathscr{g}'$, the curved coset $\mathscr{g}'\bullet J$ is well-defined if and only if $J$ contains $\Hol_{\mathscr{g}'}(\mathscr{G},\omega)=(\gamma_{{}_G}(1)h')^{-1}\Hol_\mathscr{g}(\mathscr{G},\omega)\gamma_{{}_G}(1)h'$. In other words, $\mathscr{g}'\bullet J$ is well-defined for all $\mathscr{g}'\in\mathscr{G}$ if and only if $\Hol_\mathscr{g}(\mathscr{G},\omega)$ is contained in every subgroup that is conjugate to $J$ by an element of the form $\gamma_{{}_G}(1)h'$. Because such conditions on the holonomy can be quite restrictive, we will also introduce a weaker notion of curved coset.


\begin{definition}Suppose $(\mathscr{G},\omega)$ is a Cartan geometry of type $(G,H)$ and $\mathscr{g}\in\mathscr{G}$. Let $J\leq G$ be a subgroup such that every path component of $J$ intersects $H$. Then, the \emph{geometric leaf} $\mathscr{g}\circ J$ through $\mathscr{g}$ is given by \[\mathscr{g}\circ J:=\left\{\gamma(1)h:~\begin{matrix}h\in H\cap J\text{ and }\gamma:[0,1]\to\mathscr{G}\text{ such} \\ \text{that }\gamma(0)=\mathscr{g}\text{ and }\gamma_{{}_G}([0,1])\subseteq J\end{matrix}\right\}.\]\end{definition}

Similar to the above, we will say that $\mathscr{g}\circ J$ is \emph{well-defined} if and only if every path component of $J$ intersects $H$ and $\Omega^\omega_{\mathscr{g}'}(\mathfrak{j}\wedge\mathfrak{j})\subseteq\mathfrak{j}$ for all $\mathscr{g'}\in\mathscr{g}\circ J$, where $\Omega^\omega$ is the curvature of $(\mathscr{G},\omega)$. Again, we could do without this assumption, but there is no real benefit to doing so: for $J'$ the subgroup of $J$ generated by $J\cap H$ and the identity path component $J^\circ$, $\mathscr{g}\circ J=\mathscr{g}\circ J'$, and since $\Omega^\omega(X\wedge Y)=[X,Y]-\omega([\omega^{-1}(X),\omega^{-1}(Y)])$, the condition $\Omega^\omega_{\mathscr{g}'}(\mathfrak{j}\wedge\mathfrak{j})\subseteq\mathfrak{j}$ for all $\mathscr{g}'\in\mathscr{g}\circ J$ just says that $\mathscr{g}\circ J$ is a union of integral manifolds for the distribution $\omega^{-1}(\mathfrak{j})$. Thus, unless we indicate otherwise, we will assume that our geometric leaves are well-defined whenever we refer to them.

As an example, for every $K\leq H$, $\mathscr{g}\circ K=\mathscr{g}K=\{\mathscr{g}h:h\in K\}$ is always well-defined. Moreover, if we are working over the Klein geometry $(G,\MC{G})$, then the geometric leaf $g\circ J$ will be the coset $gJ'$, where $J'$ is the subgroup generated by the identity path component $J^\circ$ and $J\cap H$, so since we assume $g\circ J$ is well-defined, $J'=J$ and $g\circ J=gJ=\{gj:j\in J\}$.


From their definitions, it is clear that $\mathscr{g}\bullet J$ and $\mathscr{g}\circ J$ are initial submanifolds of $\mathscr{G}$ when they are well-defined and $J$ is a closed subgroup of $G$, and they inherit the structure of a Cartan geometry of type $(J,H\cap J)$ over the initial submanifolds $q_{{}_{H\cap J}}(\mathscr{g}\bullet J)$ and $q_{{}_{H\cap J}}(\mathscr{g}\circ J)$ of $\mathscr{G}/(H\cap J)$, respectively, by pulling back by the inclusion map. Moreover, if $\mathscr{g}h\in\mathscr{g}\bullet J$ for some $h\in H$, then by definition developments of $\mathscr{g}h$ from $\mathscr{g}$ are contained in $J$, so $h\in H\cap J$. Thus, we can naturally identify $q_{{}_{H\cap J}}(\mathscr{g}\bullet J)$ with $q_{{}_H}(\mathscr{g}\bullet J)$.

We can actually say a bit more about the topology of $\mathscr{g}\bullet J$.

\begin{proposition}Suppose $(\mathscr{G},\omega)$ is a Cartan geometry of type $(G,H)$. Let $\mathscr{g}\in\mathscr{G}$ and let $J\leq G$ be a closed subgroup containing $\Hol_\mathscr{g}(\mathscr{G},\omega)$. Then, $\mathscr{g}\bullet J$ is an embedded closed submanifold of $\mathscr{G}$.\end{proposition}
\begin{proof}We will show that $\mathscr{g}\bullet J$ is closed because $J$ is closed, and then we will show that, just as for Lie groups, $\mathscr{g}\bullet J$ being closed implies that $\mathscr{g}\bullet J$ is an embedded submanifold.

For a subset $U\subseteq\mathfrak{g}$, we will use the notation $\exp(\omega^{-1}(U))\mathscr{g}'$ to denote the subset $\{\exp(\omega^{-1}(X))\mathscr{g}':X\in U\}\subseteq\mathscr{G}$. Note that, for $U$ open and sufficiently small, $\exp(\omega^{-1}(U))\mathscr{g}'$ is open. Suppose that, for every open $U\subseteq\mathfrak{g}$ containing $0$, $\exp(\omega^{-1}(U))\mathscr{g}'\cap\mathscr{g}\bullet J\neq\emptyset$. Let $g'\in G$ be a development of $\mathscr{g}'$ from $\mathscr{g}$. Then, elements of $\exp(\omega^{-1}(U))\mathscr{g}'$ have developments from $\mathscr{g}$ in $g'\exp(U)$, so since $\exp(\omega^{-1}(U))\mathscr{g}'\cap\mathscr{g}\bullet J\neq\emptyset$, $g'\exp(U)\cap J\neq\emptyset$ for every open $U\subseteq\mathfrak{g}$ containing $0$. Because $J$ is closed, this implies $g'\in J$, hence $\mathscr{g}'\in\mathscr{g}\bullet J$. Thus, $\mathscr{g}\bullet J$ is closed.

The rest of the proof is essentially the same as the standard proof that closed subgroups are embedded submanifolds. Suppose, by way of contradiction, that for some $\mathscr{g}'\in\mathscr{g}\bullet J$, $\exp(\omega^{-1}(U))\mathscr{g}'\cap\mathscr{g}\bullet J\not\subseteq\exp(\omega^{-1}(U\cap\mathfrak{j}))\mathscr{g}'$ for all open neighborhoods $U$ of $0$ in $\mathfrak{g}$. Consider a sequence $U_1\supseteq U_2\supseteq\cdots$ of neighborhoods of $0$ in $\mathfrak{g}$ such that $\cap_{i=0}^\infty U_i=\{0\}$. Then, for each $i$, we get an $\mathscr{a}_i\in(\exp(\omega^{-1}(U_i))\mathscr{g}'\cap\mathscr{g}\bullet J)\setminus\exp(\omega^{-1}(U_i\cap\mathfrak{j}))\mathscr{g}'$ and, decomposing $\mathfrak{g}$ (as a vector space) as $\mathfrak{g}=V\oplus\mathfrak{j}$, we can choose $Z_i\in\mathfrak{j}$ and $v_i\in V\setminus\{0\}$ such that $Z_i\rightarrow 0$, $v_i\rightarrow 0$, and $\mathscr{a}_i=\exp(\omega^{-1}(Z_i))\exp(\omega^{-1}(v_i))\mathscr{g}'$ for sufficiently small $U_i$. Since $\mathscr{a}_i\in\mathscr{g}\bullet J$, we have that $\exp(\omega^{-1}(v_i))\mathscr{g}'=\exp(\omega^{-1}(Z_i))^{-1}\mathscr{a}_i\in\mathscr{g}\bullet J$.

Give $V$ a norm $\|\cdot\|$. Since $v_i$ is nonzero for each $i$, we can consider $\frac{1}{\|v_i\|}v_i$, which subconverges to some $v$ of norm $1$ in $V$ because the unit sphere in $V$ is compact. Since \[\left\lfloor\frac{t}{\|v_i\|}\right\rfloor v_i=\left(\left\lfloor\frac{t}{\|v_i\|}\right\rfloor\|v_i\|\right)\left(\frac{1}{\|v_i\|}v_i\right)\rightarrow tv,\] where $\lfloor\cdot\rfloor$ is the floor function, it follows that $\exp(\lfloor\frac{t}{\|v_i\|}\rfloor\omega^{-1}(v_i))\mathscr{g}'$ subconverges to $\exp(t\omega^{-1}(v))\mathscr{g}'$. As shown above, $\exp(\omega^{-1}(v_i))\mathscr{g}'\in\mathscr{g}\bullet J$, hence $\exp(v_i)\in J$, so \[\exp\left(\left\lfloor\frac{t}{\|v_i\|}\right\rfloor\omega^{-1}(v_i)\right)\mathscr{g}'=\exp(\omega^{-1}(v_i))^{\lfloor\frac{t}{\|v_i\|}\rfloor}\mathscr{g}'\in\mathscr{g}\bullet J.\] But $\mathscr{g}\bullet J$ is closed, so $\exp(t\omega^{-1}(v))\mathscr{g}'\in\mathscr{g}\bullet J$ for all $t$, so $v\in V\cap\mathfrak{j}=\{0\}$, which contradicts $\|v\|=1$. Thus, for all $\mathscr{g}'\in\mathscr{g}\bullet J$, $\exp(\omega^{-1}(U))\mathscr{g}'\cap\mathscr{g}\bullet J\subseteq\exp(\omega^{-1}(U\cap\mathfrak{j}))\mathscr{g}'$ for some open neighborhood $U$ of $0$ in $\mathfrak{g}$, so $\mathscr{g}\bullet J$ is an embedded submanifold.\mbox{\qedhere}\end{proof}


If we fix a point $\mathscr{e}\in\mathscr{G}$ that we think of as the ``identity element", then both $\mathscr{e}\bullet J$ and $\mathscr{e}\circ J$ can reasonably be thought of as ``subgroups" of $(\mathscr{G},\omega)$ when they are well-defined. Moreover, once we have chosen an ``identity element" for $(\mathscr{G},\omega)$, we can assign developments from that point, and using this, we can define \emph{curved right cosets} by analogy to $Jg=g(g^{-1}Jg)$.


\begin{definition}Suppose $(\mathscr{G},\omega)$ is a Cartan geometry of type $(G,H)$. Fix a point $\mathscr{e}\in\mathscr{G}$, and let $J\leq G$ be a subgroup containing $\Hol_\mathscr{e}(\mathscr{G},\omega)$. For $g\in G$ a development of $\mathscr{g}$ from $\mathscr{e}$, the \emph{curved (right) coset} $J\bullet\mathscr{g}$ through $\mathscr{g}\in\mathscr{G}$ is given by \[J\bullet\mathscr{g}:=\mathscr{g}\bullet(g^{-1}Jg).\]\end{definition}


Here, again, we might as well always require that $g^{-1}Jg$ contains $\Hol_\mathscr{g}(\mathscr{G},\omega)$, since nothing is gained by assuming otherwise. However, since $\Hol_\mathscr{e}(\mathscr{G},\omega)=g\Hol_\mathscr{g}(\mathscr{G},\omega)g^{-1}$, this just means that $J$ has to contain the holonomy group at $\mathscr{e}$, and if this holds, then curved right cosets through $\mathscr{g}$ are well-defined independent of any condition on $\mathscr{g}$.

Note that, in order to define these curved right cosets, we must choose a basepoint $\mathscr{e}\in\mathscr{G}$, so by Proposition \ref{basepoint}, if $J$ is closed, then the data needed to define a curved right coset is precisely the data needed to define a holonomy reduction. In terms of the holonomy reduction $\sigma:\mathscr{G}\to G/J$ induced by the choice of basepoint, if $g$ is a development of $\mathscr{g}$, then \begin{align*}J\bullet\mathscr{g} & =\mathscr{g}\bullet(g^{-1}Jg) \\ & =\left\{\beta(1)h:~\begin{matrix}h\in H\text{ and }\beta:[0,1]\to\mathscr{G}\text{ such that} \\ \beta(0)=\mathscr{g}\text{ and }\beta_{{}_G}(1)h\in g^{-1}Jg\end{matrix}\right\} \\ & =\left\{\gamma(1)h:~\begin{matrix}h\in H\text{ and }\gamma:[0,1]\to\mathscr{G}\text{ such that} \\ \gamma(0)=\mathscr{e}\text{ and }\gamma_{{}_G}(1)h\in Jg\end{matrix}\right\} \\ & =\sigma^{-1}(g^{-1}J).\end{align*} Moreover, $q_{{}_H}(J\bullet\mathscr{g})=q_{{}_H}((J\bullet\mathscr{g})H)=q_{{}_H}(\sigma^{-1}(Hg^{-1}J))$, and $q_{{}_H}(\sigma^{-1}(Hg^{-1}J))$ is the set of all points on the base manifold with $H$-type $Hg^{-1}J$, so curved orbits are just the projections of curved right cosets to the base manifold.

Interestingly, curved cosets are so well-behaved that we can even consider the corresponding quotient manifolds.

\begin{proposition}Suppose $(\mathscr{G},\omega)$ is a Cartan geometry of type $(G,H)$. Fix $\mathscr{e}\in\mathscr{G}$ and a closed subgroup $J\leq G$ containing $\Hol_\mathscr{e}(\mathscr{G},\omega)$. Consider the map $\pi_{{}_J}:\mathscr{G}\to J\backslash\mathscr{G}, \mathscr{g}\mapsto J\bullet\mathscr{g}$, where $J\backslash\mathscr{G}$ is the space of curved right cosets under the quotient topology given by $\pi_{{}_J}$. Then, $J\backslash\mathscr{G}$ admits a (necessarily unique) smooth manifold structure such that $\pi_{{}_J}$ is a submersion.\end{proposition}
\begin{proof}We will use the notation $J\bullet N=\pi_{{}_J}^{-1}(\pi_{{}_J}(N))=\cup_{\mathscr{g}\in N}J\bullet\mathscr{g}$ for subsets $N\subseteq\mathscr{G}$. Given an open neighborhood $U$ of $0$ in $\mathfrak{g}$ that is small enough that $\exp(\omega^{-1}(U))\mathscr{g}$ is open and $\exp(\omega^{-1}(u))\mathscr{g}$ is well-defined for all $u\in U$, consider $\mathscr{g}'\in J\bullet(\exp(\omega^{-1}(U))\mathscr{g})$. Suppose $g$ is a development of $\mathscr{g}$ and $g'$ is a development of $\mathscr{g}'$. For sufficiently small $U'$ around $0$ in $\mathfrak{g}$ such that $\exp(\omega^{-1}(U'))\mathscr{g}'$ is open and $g'\exp(U')\subseteq Jg\exp(U)$, we get $\exp(\omega^{-1}(U'))\mathscr{g}'\subseteq J\bullet(\exp(\omega^{-1}(U))\mathscr{g})$, so $\mathscr{g}'$ is an interior point. This means that $J\bullet(\exp(\omega^{-1}(U))\mathscr{g})=\pi_{{}_J}^{-1}(\pi_{{}_J}(\exp(\omega^{-1}(U))\mathscr{g}))$ is open in $\mathscr{G}$, so since $J\backslash\mathscr{G}$ has the quotient topology, $\pi_{{}_J}(\exp(\omega^{-1}(U))\mathscr{g})$ is open in $J\backslash\mathscr{G}$. Thus, $\pi_{{}_J}$ is an open map.

Suppose $g$ is a development of $\mathscr{g}$ from $\mathscr{e}$ and $g'$ is a development of $\mathscr{g}'$ from $\mathscr{e}$. If $J\bullet(\exp(\omega^{-1}(U))\mathscr{g})\cap J\bullet(\exp(\omega^{-1}(V))\mathscr{g}')\neq\emptyset$ for all neighborhoods $U$ and $V$ of $0$ in $\mathfrak{g}$, then the corresponding developments $Jg\exp(U)\cap Jg'\exp(V)\neq\emptyset$ for all neighborhoods $U$ and $V$ of $0$ in $\mathfrak{g}$, so since $J$ is closed, $Jg=Jg'$, so $J\bullet\mathscr{g}=J\bullet\mathscr{g}'$. Thus, $J\backslash\mathscr{G}$ is Hausdorff.

The rest of this proof is essentially the same as the standard proof for the homogeneous space $J\backslash G$. Pick a complementary subspace $V$ to $\mathfrak{j}$ in $\mathfrak{g}$ and $W$ a small enough neighborhood of $0$ in $V$ so that, for some neighborhood $U$ of $0$ in $\mathfrak{j}$, $\psi:U\times W\to\psi(U\times W)\subseteq\mathscr{G},~(u,w)\mapsto\exp(\omega^{-1}(w))\exp(\omega^{-1}(u))\mathscr{e}$ is an embedding. Thus, $\pi_{{}_J}\circ\psi|_{\{0\}\times W}$ is bijective, continuous, and open, so its inverse is a chart around $J\bullet\mathscr{e}$, and since we can push $\psi$ around all of $\mathscr{G}$ using flows by $\omega$-constant vector fields and right-translating by elements of $H$, this gives a smooth manifold structure on $J\backslash\mathscr{G}$.\mbox{\qedhere}\end{proof}

In particular, if $(\mathscr{G},\omega)$ is a Cartan geometry of type $(G,H)$, $\mathscr{e}\in\mathscr{G}$, and $J$ is a closed normal subgroup of $G$ containing $\Hol_\mathscr{e}(\mathscr{G},\omega)$, then $J\backslash\mathscr{G}$ is a principal $(H\cap J)\backslash H$-bundle and naturally inherits a Cartan connection $\bar{\omega}$ of type $(J\backslash G,(H\cap J)\backslash H)$ with trivial holonomy, given by $\pi_{{}_J}^*\bar{\omega}=\omega+\mathfrak{j}$.


\section{Relation to tractor connections}\label{tractorconnections}
When we want to study Lie groups, it is often useful to look at their representations. Naturally, then, when we want to study Cartan geometries, it might be useful to have a corresponding notion of representation. To that end, suppose we are given a Cartan geometry $(\mathscr{G},\omega)$ of type $(G,H)$ over a manifold $M$ and a finite-dimensional representation $\rho:G\to \GLin(V)$.

We would like to find a ``representation'' for $(\mathscr{G},\omega)$ corresponding to $\rho$. If we fix an ``identity element'' $\mathscr{e}\in\mathscr{G}$ and $\Hol_\mathscr{e}(\mathscr{G},\omega)\leq\ker(\rho)$, then we have a convenient answer: define $\rho(\mathscr{g}):=\rho(g)$ for each $\mathscr{g}\in\mathscr{G}$, where $g\in G$ is some development of $\mathscr{g}$ from $\mathscr{e}$. Unfortunately, this is generally not a useful assumption to make for $(\mathscr{G},\omega)$; if $G$ is a connected centerless simple Lie group, for example, then the only normal subgroups of $G$ are $\{e\}$ and $G$ itself, so if $\Hol_\mathscr{e}(\mathscr{G},\omega)$ is nontrivial, then this will only get us the trivial representation. Thus, while this is a good answer when it works, we would like something more general that does not have such stringent requirements on the holonomy.

Of course, the restriction $\rho|_H$ of $\rho$ to $H$ is still a representation. Thus, identifying $H$ with the ``subgroup" $\mathscr{e}\circ H=\mathscr{e}H$, we can try to form the ``(co)induced representation" of $\rho|_H$ to $(\mathscr{G},\omega)$, given by the space $\Gamma(\mathscr{G}\times_H V)$ of sections of the associated bundle $\mathscr{G}\times_H V$ over $M$, which we identify with the space of $H$-equivariant functions from $\mathscr{G}$ to $V$. In the model case, $G$ acts on $\Gamma(G\times_H V)$ by $g\cdot\sigma:=\sigma\circ\Lt{g^{-1}}$, and an isomorphic copy of the representation $\rho$ is contained in $\Gamma(G\times_H V)$, given by the subspace of sections of the form $\widetilde{v}:G\to V, g\mapsto\rho(g)^{-1}v$ with $v\in V$. This gives an identification of $G\times_H V$ with the trivial bundle $G/H\times V$, where the copy of $\rho$ in $\Gamma(G\times_H V)$ is precisely the subspace that is identified with the constant sections of $G/H\times V$.

In the general case, we lack a meaningful notion of action of $(\mathscr{G},\omega)$ on $\Gamma(\mathscr{G}\times_H V)$, and we do not necessarily have an isomorphism between $\mathscr{G}\times_H V$ and $M\times V$. However, under the identification of $G/H\times V$ and $G\times_H V$ in the model case, the exterior derivative for $G/H\times V$ determines a corresponding differential operator $\nabla^\rho:\Gamma(G\times_H V)\to \Gamma(T^\vee(G/H)\otimes (G\times_H V))$ given by, for $X\in T(G/H)$ and $\sigma\in\Gamma(G\times_H V)$, \[(\nabla^\rho\sigma)(X)=\nabla^\rho_X\sigma=\widehat{X}(\sigma)+\rho_*(\MC{G}(\widehat{X}))\sigma,\] where $\widehat{X}\in TG$ is an arbitrary lift of $X$ to $TG$ and $\rho_*:\mathfrak{g}\to\mathfrak{gl}(V)$ is the Lie algebra homomorphism induced by $\rho$, and the kernel of this differential operator is the copy of $\rho$. We can define a corresponding differential operator $\nabla^\rho:\Gamma(\mathscr{G}\times_H V)\to\Gamma(T^\vee M\otimes(\mathscr{G}\times_H V))$ given by, for $X\in TM$ and $\sigma\in\Gamma(\mathscr{G}\times_H V)$, \[(\nabla^\rho\sigma)(X)=\nabla^\rho_X\sigma:=\widehat{X}(\sigma)+\rho_*(\omega(\widehat{X}))\sigma,\] where $\widehat{X}\in T\mathscr{G}$ is an arbitrary lift of $X$ to $T\mathscr{G}$ and $\rho_*$ is, again, the Lie algebra homomorphism induced by $\rho$. This operator $\nabla^\rho$ is better known as the \emph{tractor connection} for $\rho$.

\begin{definition}Let $\rho:G\to\GLin(V)$ be a representation of $G$. Given a Cartan geometry $(\mathscr{G},\omega)$ of type $(G,H)$ over $M$, the \emph{tractor bundle} corresponding to $\rho$ is defined to be the associated bundle $\mathscr{G}\times_H V$. Identifying sections $\sigma$ of $\mathscr{G}\times_H V$ with $H$-equivariant functions $\sigma:\mathscr{G}\to V$, the \emph{tractor connection} $\nabla^\rho$ is defined by \[\nabla^\rho_X\sigma:=\widehat{X}(\sigma)+\rho_*(\omega(\widehat{X}))\sigma,\] where $X\in TM$ and $\widehat{X}$ is some lift of $X$ to $T\mathscr{G}$.\end{definition}

Not surprisingly, tractor connections are valuable tools for working with Cartan geometries, in much the same way that representation theory is useful for studying Lie groups. As an example, for (torsion-free) Cartan geometries of type $(\Iso(m),\Orth(m))$, which correspond to Riemannian manifolds, the tractor bundle corresponding to the standard representation of $\Iso(m)\simeq\mathbb{R}^m\rtimes\Orth(m)$ on $\mathbb{R}^m$ is just the tangent bundle of the base manifold, and the tractor connection is the Levi-Civita connection.

For parabolic geometries, the topic of tractor connections is covered extensively in \cite{CapSlovak}, and their holonomy (defined in the usual way for vector bundle connections) has been thoroughly studied in several cases (see \cite{Armstrong}, \cite{Armstrong1}, and \cite{Armstrong2}). Indeed, when thinking of the holonomy group of a Cartan geometry in the extrinsic sense, as the holonomy group of a principal connection on an extended principal bundle, most of the results of this section are almost certainly considered well-known in some form or another.

We can think of either the tractor connection $\nabla^\rho$ itself or its space of parallel sections as the analogue of $\rho$ for $(\mathscr{G},\omega)$. Indeed, we can explicitly describe these parallel sections in a way precisely analogous to the copy of $\rho$ in $\Gamma(G\times_H V)$.

\begin{lemma}\label{paralleltransport} Suppose $(\mathscr{G},\omega)$ is a Cartan geometry of type $(G,H)$ over $M$ and $\nabla^\rho$ is the tractor connection on $\mathscr{G}\times_H V$ associated to a representation $\rho:G\to\GLin(V)$. Fix a point $(\mathscr{e},v)\in\mathscr{G}\times_H V$ and a path $\gamma:[0,1]\to M$ with $\gamma(0)=q_{{}_H}(\mathscr{e})$. For $\widehat{\gamma}:[0,1]\to\mathscr{G}$ a lift of $\gamma$ to $\mathscr{G}$ with $\widehat{\gamma}(0)=\mathscr{e}$, the parallel transport of $(\mathscr{e},v)$ along $\gamma$ with respect to $\nabla^\rho$ is given by $\|_\gamma(\mathscr{e},v):t\mapsto(\widehat{\gamma}(t),\rho(\widehat{\gamma}_{{}_G}(t))^{-1}v)$.\end{lemma}
\begin{proof}Using the usual abuses of notation for computing parallel transport, \begin{align*}\nabla^\rho_{\dot{\gamma}}\|_\gamma(\mathscr{e},v) & =\partial_t(\rho(\widehat{\gamma}_{{}_G})^{-1}v)+\rho_*(\omega(\dot{\widehat{\gamma}}))(\rho(\widehat{\gamma}_{{}_G})^{-1}v) \\ & =-\rho_*(\omega(\dot{\widehat{\gamma}}))(\rho(\widehat{\gamma}_{{}_G})^{-1}v)+\rho_*(\omega(\dot{\widehat{\gamma}}))(\rho(\widehat{\gamma}_{{}_G})^{-1}v)=0.\quad\mbox{\qedhere}\end{align*}\end{proof}

\begin{corollary}A section $\sigma\in\Gamma(\mathscr{G}\times_H V)$ is parallel with respect to $\nabla^\rho$ if and only if the corresponding $H$-equivariant function $\sigma:\mathscr{G}\to V$ is of the form $\sigma:\mathscr{g}\mapsto\rho(\mathscr{g})^{-1}\sigma(\mathscr{e})$, where $\rho(\mathscr{g})$ is defined to be $\rho(g)$ for some development $g\in G$ of $\mathscr{g}$ from $\mathscr{e}$.\end{corollary}

Using Lemma \ref{paralleltransport}, we also get an explicit description of the holonomy of a tractor connection in terms of the holonomy of the associated Cartan geometry.

\begin{corollary}\label{holconversion} The isomorphism between the fiber of $\mathscr{G}\times_HV$ over $q_{{}_H}(\mathscr{e})$ and $V$ given by identifying $(\mathscr{e}h,\rho(h)^{-1}v)=(\mathscr{e},v)$ with $v$ induces an isomorphism \[\Hol_{q_{{}_H}(\mathscr{e})}(\mathscr{G}\times_HV,\nabla^\rho)\simeq\rho(\Hol_\mathscr{e}(\mathscr{G},\omega)).\]\end{corollary}

It was observed in \cite{CGH} that parallel sections of tractor bundles determine holonomy reductions. In our current formulation, this just amounts to the fact that, for $\widetilde{v}:\mathscr{g}\mapsto\rho(\mathscr{g})^{-1}v$ to be well-defined, we only need $\widetilde{v}_\mathscr{g}=\rho(g)^{-1}v=\rho(jg)^{-1}v=\rho(g)^{-1}\rho(j)^{-1}v$ for every $j\in\Hol_\mathscr{e}(\mathscr{G},\omega)$ and $g$ a development of $\mathscr{g}$ from $\mathscr{e}$, so $\widetilde{v}$ is well-defined if and only if $\Hol_\mathscr{e}(\mathscr{G},\omega)$ is contained in the stabilizer $\mathrm{Stab}_G(v)$. In particular, if $\Hol_\mathscr{e}(\mathscr{G},\omega)\leq\ker(\rho)$, then we can recover the full copy of $V$ from the beginning of this section.


\section{Interactions with automorphisms}\label{automorphisms}
More than being just a principal $H$-bundle on which a Cartan connection of type $(G,H)$ is defined, in the model case the Lie group $G$ is also the group of symmetries for the geometric structure, acting on itself and the underlying homogeneous space by left-translation. For general Cartan geometries, geometric automorphisms act as analogues of these left-translations. In this section, we will briefly recall some elementary properties of automorphisms of Cartan geometries and describe how they interact with curved cosets and holonomy reductions.

\begin{definition}Given two Cartan geometries $(\mathscr{G},\omega)$ and $(\mathscr{G}',\omega')$ of type $(G,H)$, an $H$-equivariant map $\varphi:\mathscr{G}\to\mathscr{G}'$ is a \emph{geometric map} if and only if $\varphi^*\omega'=\omega$. A \emph{geometric isomorphism} is a bijective geometric map, and a \emph{(geometric) automorphism} is a geometric isomorphism from a Cartan geometry to itself.\end{definition}

We will denote by $\Aut(\mathscr{G},\omega)$ the group of automorphisms of $(\mathscr{G},\omega)$. Similar to how $a(gb)=(ag)b$ for $a,b,g\in G$, for $\varphi\in\Aut(\mathscr{G},\omega)$, $h\in H$, $X\in\mathfrak{g}$, and $\mathscr{g}\in\mathscr{G}$ we get $\varphi(\mathscr{g}h)=\varphi(\mathscr{g})h$ by $H$-equivariance and $\varphi(\exp(t\omega^{-1}(X))\mathscr{g})=\exp(t\omega^{-1}(X))\varphi(\mathscr{g})$ because $\varphi^*\omega=\omega$. For a Cartan geometry over a connected base manifold, if we fix a point $\mathscr{e}\in\mathscr{G}$, then every element of $\mathscr{G}$ is of the form $(\exp(\omega^{-1}(X_k))\cdots\exp(\omega^{-1}(X_1))\mathscr{e})h$ for some $X_1,\dots,X_k\in\mathfrak{g}$ and $h\in H$, so automorphisms $\varphi\in\Aut(\mathscr{G},\omega)$ are uniquely determined by the image $\varphi(\mathscr{e})$ of $\mathscr{e}$. In particular, $\Aut(\mathscr{G},\omega)$ acts freely on $\mathscr{G}$. We can further use this analogue of associativity to see that the orbits of $\Aut(\mathscr{G},\omega)$ are embedded closed submanifolds, so $\Aut(\mathscr{G},\omega)$ inherits a natural Lie group structure from its orbits on $\mathscr{G}$.

For $\varphi\in\Aut(\mathscr{G},\omega)$ and $\gamma:[0,1]\to\mathscr{G}$ a path in $\mathscr{G}$, $(\varphi\circ\gamma)^*\omega=\gamma^*\varphi^*\omega=\gamma^*\omega$, so $\gamma_{{}_G}=(\varphi\circ\gamma)_{{}_G}$. Therefore, if $\mathscr{e},\mathscr{g}\in\mathscr{G}$, $g\in G$ is a development of $\mathscr{g}$ from $\mathscr{e}$, and $a\in G$ is a development of $\varphi(\mathscr{e})$ from $\mathscr{e}$, then $ag$ is a development of $\varphi(\mathscr{g})$ from $\mathscr{e}$. Heuristically, this means that automorphisms look like particular left-translations to the developments.

Using the holonomy group, we can even place fairly strong restrictions on what left-translations an automorphism can look like in this sense.

\begin{proposition}\label{holaut} Suppose $(\mathscr{G},\omega)$ is a Cartan geometry of type $(G,H)$ and $\mathscr{e}\in\mathscr{G}$. Let $\varphi\in\Aut(\mathscr{G},\omega)$. Then, $\varphi(\mathscr{e})\in\mathscr{e}\bullet\mathrm{N}_G(\Hol_\mathscr{e}(\mathscr{G},\omega))$, where $\mathrm{N}_G(J)$ denotes the normalizer of a subgroup $J\leq G$. In other words, if $\varphi(\mathscr{e})$ develops to $a\in G$ from $\mathscr{e}$, then $a\in\mathrm{N}_G(\Hol_\mathscr{e}(\mathscr{G},\omega))$.\end{proposition}
\begin{proof}If $\varphi\in\Aut(\mathscr{G},\omega)$, then \begin{align*}\Hol_{\varphi(\mathscr{e})}(\mathscr{G},\omega) & =\left\{\alpha_{{}_G}(1)h\in G: \begin{matrix}h\in H\text{ and }\alpha:[0,1]\to\mathscr{G}\\ \text{ such that }\alpha(0)=\varphi(\mathscr{e})=\alpha(1)h\end{matrix}\right\} \\ & =\left\{(\varphi(\beta))_{{}_G}(1)h\in G: \begin{matrix}h\in H\text{ and }\beta:[0,1]\to\mathscr{G}\\ \text{ such that }\beta(0)=\mathscr{e}=\beta(1)h\end{matrix}\right\} \\ & =\Hol_\mathscr{e}(\mathscr{G},\omega).\end{align*} Thus, if $\varphi(\mathscr{e})$ develops to $a$ from $\mathscr{e}$, then \[\Hol_\mathscr{e}(\mathscr{G},\omega)=\Hol_{\varphi(\mathscr{e})}(\mathscr{G},\omega)=a^{-1}\Hol_\mathscr{e}(\mathscr{G},\omega)a.~~\mbox{\qedhere}\]\end{proof}

In particular, if a Cartan geometry has a large automorphism group, then the normalizer of the holonomy group must also be large. Further, if a Cartan geometry of type $(G,H)$ is homogeneous, meaning that the automorphism group acts transitively on the base manifold, then the normalizer of the holonomy group will have an open orbit in $G/H$.

As we might expect, our curved cosets behave well under automorphisms: for $\varphi\in\Aut(\mathscr{G},\omega)$, $\mathscr{e}\in\mathscr{G}$, and $J\leq G$ a subgroup containing $\Hol_\mathscr{e}(\mathscr{G},\omega)$, \begin{align*}\varphi(\mathscr{e}\bullet J) & =\left\{\varphi(\gamma(1)h):~\begin{matrix}h\in H\text{ and }\gamma:[0,1]\to\mathscr{G}\text{ such} \\ \text{that }\gamma(0)=\mathscr{e}\text{ and }\gamma_{{}_G}(1)h\in J\end{matrix}\right\} \\ & =\left\{\beta(1)h:~\begin{matrix}h\in H\text{ and }\beta:[0,1]\to\mathscr{G}\text{ such} \\ \text{that }\beta(0)=\varphi(\mathscr{e})\text{ and }\beta_{{}_G}(1)h\in J\end{matrix}\right\} \\ & = \varphi(\mathscr{e})\bullet J,\end{align*} in the same way that $a(eJ)=aJ$. However, analogous to how $a(Jg)$ is not necessarily equal to $J(ag)$ unless $a\in \mathrm{N}_G(J)$, $\varphi(J\bullet\mathscr{g})$ is not necessarily equal to $J\bullet\varphi(\mathscr{g})$ unless an analogous condition is met.

\begin{theorem}\label{normalizer} Suppose $(\mathscr{G},\omega)$ is a Cartan geometry of type $(G,H)$, $\mathscr{e}\in\mathscr{G}$, and $J\leq G$ is a subgroup containing $\Hol_\mathscr{e}(\mathscr{G},\omega)$. For $\varphi\in\Aut(\mathscr{G},\omega)$, if $\varphi(\mathscr{e})\in\mathscr{e}\bullet \mathrm{N}_G(J)$, then $\varphi(J\bullet\mathscr{g})=J\bullet\varphi(\mathscr{g})$ for all $\mathscr{g}\in\mathscr{G}$.\end{theorem}
\begin{proof}Suppose $a$ is a development of $\varphi(\mathscr{e})$ from $\mathscr{e}$ and $g$ is a development of $\mathscr{g}$ from $\mathscr{e}$. Then, $\varphi(\mathscr{g})$ will have development $ag$, so \[J\bullet\varphi(\mathscr{g})=\varphi(\mathscr{g})\bullet((ag)^{-1}J(ag))=\varphi(\mathscr{g})\bullet(g^{-1}(a^{-1}Ja)g)=\varphi((a^{-1}Ja)\bullet\mathscr{g}).\] Thus, if $a^{-1}Ja=J$, then $J\bullet\varphi(\mathscr{g})=\varphi(J\bullet\mathscr{g})$. Note that $a^{-1}Ja$ necessarily contains $\Hol_\mathscr{e}(\mathscr{G},\omega)$ by Proposition \ref{holaut}, so $a^{-1}Ja\bullet\mathscr{g}$ is well-defined.\mbox{\qedhere}\end{proof}

By Proposition \ref{holaut}, the above theorem implies that \[\varphi(\Hol_\mathscr{e}(\mathscr{G},\omega)\bullet\mathscr{g})=\Hol_\mathscr{e}(\mathscr{G},\omega)\bullet\varphi(\mathscr{g}).\]

Since $a^{-1}Ja$ and $J$ are not disjoint, it is possible for $a^{-1}Ja\bullet\mathscr{g}=J\bullet\mathscr{g}$, and hence $\varphi(J\bullet\mathscr{g})=J\bullet\varphi(\mathscr{g})$, without $\varphi(\mathscr{e})\in\mathscr{e}\bullet\mathrm{N}_G(J)$. For example, this occurs for the Cartan geometry of type $(\mathbb{R}^2\rtimes\Orth(2),\Orth(2))$ inherited by an open unit disk around the origin in the Euclidean plane: for $J$ the subgroup generated by a translation that moves the origin outside of the unit disk, $\varphi$ a rotation around the origin by an angle that is not an integer multiple of $\pi$, and $\mathscr{e}$ corresponding to the image of the identity element of $\mathbb{R}^2\rtimes\Orth(2)$ under the restriction to the disk, we get that $\varphi(\mathscr{e})=\mathscr{e}A$, where $A$ is the element of $\Orth(2)$ corresponding to the chosen rotation around the origin, but $A$ does not normalize $J$. However, we still get $\varphi(J\bullet\mathscr{e})=J\bullet\varphi(\mathscr{e})=\{\varphi(\mathscr{e})\}$, essentially because the geometry does not ``see'' the translations in $J$ from $\mathscr{e}$.

In contrast, if $aJg\neq Jg$, then $aJg$ and $Jg$ are disjoint, so the corresponding condition for curved orbits is necessary and sufficient.

\begin{theorem}Suppose $(\mathscr{G},\omega)$ is a Cartan geometry of type $(G,H)$, $\mathscr{e}\in\mathscr{G}$, and $J\leq G$ is a subgroup containing $\Hol_\mathscr{e}(\mathscr{G},\omega)$. For $\varphi\in\Aut(\mathscr{G},\omega)$ and $\mathscr{g}\in\mathscr{G}$, $\varphi(J\bullet\mathscr{g})=J\bullet\mathscr{g}$ if and only if $\varphi(\mathscr{e})\in\mathscr{e}\bullet J$.\end{theorem}
\begin{proof}As we saw in the proof of Theorem \ref{normalizer}, with $a\in G$ a development of $\varphi(\mathscr{e})$ and $g\in G$ a development of $\mathscr{g}$, $\varphi(J\bullet\mathscr{g})=aJa^{-1}\bullet\varphi(\mathscr{g})$, so if $a\in J$, then $\varphi(J\bullet\mathscr{g})=J\bullet\mathscr{g}$. Conversely, if $\varphi(J\bullet\mathscr{g})=J\bullet\mathscr{g}$, then every element $\mathscr{g}'\in\varphi(J\bullet\mathscr{g})=J\bullet\mathscr{g}$ has developments of the form $ajg$ and $j'g$ for some $j,j'\in J$, so $ajg(j'g)^{-1}=aj(j')^{-1}\in\Hol_\mathscr{e}(\mathscr{G},\omega)$, hence $a\in J$.\mbox{\qedhere}\end{proof}

One situation where this idea applies particularly well is when the base manifold itself is a single curved orbit of $J$.

\begin{corollary}\label{essential} Suppose $q_{{}_H}(J\bullet\mathscr{e})=q_{{}_H}(\mathscr{G})$. Then, for $\mathscr{g}\in\mathscr{G}$ with development $g\in G$, the $H$-equivariant extensions of the automorphisms of the induced Cartan geometry of type $(g^{-1}Jg,g^{-1}Jg\cap H)$ on $J\bullet\mathscr{g}$ are automorphisms of $(\mathscr{G},\omega)$, and an automorphism $\varphi$ of $(\mathscr{G},\omega)$ restricts to an automorphism of $J\bullet\mathscr{g}$ if and only if $\varphi(\mathscr{e})\in\mathscr{e}\bullet J$.\end{corollary}

As we will see below, Corollary \ref{essential} could be remarkably useful, in some circumstances, for studying essential automorphisms of parabolic geometries. However, it is already quite useful on its own. It tells us, for example, that if $Q\subseteq\mathrm{Fr}(M)$ is a first-order $H$-structure, where $\mathrm{Fr}(M)$ is the frame bundle of $M$, and $\nabla$ is an affine connection compatible with $Q$, then an automorphism $\varphi:M\to M$ of the affine manifold $(M,\nabla)$ restricts to an automorphism of the first-order $H$-structure if and only if, for \textit{some} frame $u\in Q$, the derivative map $\varphi_{*q_{{}_H}(u)}$ satisfies $\varphi_{*q_{{}_H}(u)}\circ u\in Q$. In other words, $\varphi$ preserves \textit{all} of $Q$ if and only if it preserves a \textit{single} frame of $Q$.

We can also consider \emph{infinitesimal automorphisms} of $(\mathscr{G},\omega)$, which are vector fields $\xi$ on $\mathscr{G}$ whose flows are automorphisms of $(\mathscr{G},\omega)$. Equivalently, $\xi$ is an infinitesimal automorphism if and only if it is a complete vector field such that $[\xi,\omega^{-1}(X)]=0$ for all $X\in\mathfrak{g}$ and $\Rt{h*}\xi=\xi$ for all $h\in H$. We will denote the Lie algebra of infinitesimal automorphisms of $(\mathscr{G},\omega)$ under the negative of the Lie bracket by $\mathfrak{aut}(\mathscr{G},\omega)$, since this is precisely the Lie algebra of $\Aut(\mathscr{G},\omega)$.

Consider the subalgebra $\mathfrak{z}_\mathfrak{g}(\Hol_\mathscr{e}(\mathscr{G},\omega))\leq\mathfrak{g}$ of all $\Ad_{\Hol_\mathscr{e}(\mathscr{G},\omega)}$-invariant elements of $\mathfrak{g}$. For $Z\in\mathfrak{z}_\mathfrak{g}(\Hol_\mathscr{e}(\mathscr{G},\omega))$, we can consider the vector field $\widetilde{Z}:\mathscr{g}\mapsto\omega_\mathscr{g}^{-1}(\Ad_{\mathscr{g}^{-1}}Z)$, where $\Ad_{\mathscr{g}^{-1}}$ is, again, shorthand for $\Ad_{g^{-1}}$ for some development $g$ of $\mathscr{g}$. By the considerations of the previous section, these will naturally correspond to parallel sections of the adjoint tractor bundle $\mathscr{G}\times_H\mathfrak{g}$ with respect to the tractor connection $\nabla^\Ad$. Because developments of $\mathscr{g}$ are all of the form $\Hol_\mathscr{e}(\mathscr{G},\omega)g$ whenever $g$ is a development of $\mathscr{g}$, $\widetilde{Z}$ is well-defined, since if $jg$ were another development of $\mathscr{g}$, then $\Ad_{(jg)^{-1}}Z=\Ad_{g^{-1}}(\Ad_{j^{-1}}Z)=\Ad_{g^{-1}}Z$. On the model geometry $(G,H)$, these are precisely the right-invariant vector fields, so we will call the vector fields $\widetilde{Z}$ \emph{right-invariant vector fields} on $\mathscr{G}$.

Since they coincide in the model geometry, we would like infinitesimal automorphisms and right-invariant vector fields to be the same thing. We will soon see that they often are not, but as we should expect, the curvature and holonomy will tell us exactly how they can differ and constrain how badly they can fail to coincide.

Given a vector field $\eta$ on $\mathscr{G}$ and $X\in\mathfrak{g}$, \begin{align*}\Omega(\eta\wedge\omega^{-1}(X)) & =\mathrm{d}\omega(\eta\wedge\omega^{-1}(X))+[\omega(\eta),X] \\ & =\eta\omega(\omega^{-1}(X))-\omega^{-1}(X)\omega(\eta)-\omega([\eta,\omega^{-1}(X)])+[\omega(\eta),X] \\ & =-\omega^{-1}(X)\omega(\eta)-\omega([\eta,\omega^{-1}(X)])-[X,\omega(\eta)],\end{align*} so \[\omega^{-1}(X)\omega(\eta)=-\left(\ad_X\omega(\eta)+\Omega(\eta\wedge\omega^{-1}(X))+\omega([\eta,\omega^{-1}(X)])\right).\] In particular, for $\xi\in\mathfrak{aut}(\mathscr{G},\omega)$, $[\xi,\omega^{-1}(X)]=0$, so \[\omega^{-1}(X)\omega(\xi)=-(\ad_X\omega(\xi)+\Omega(\xi\wedge\omega^{-1}(X))).\] Thus, for $\gamma:[0,1]\to\mathscr{G}$ a path with $\gamma(0)=\mathscr{e}$, $\omega(\xi_{\gamma(t)})=\Ad_{\gamma_{{}_G}(t)^{-1}}\omega(\xi_\mathscr{e})$ for all $t\in[0,1]$ if and only if $\Omega(\xi_{\gamma(t)}\wedge\dot{\gamma}(t))=0$ for all $t\in[0,1]$, so $\xi$ is a right-invariant vector field if and only if $\Omega(\xi\wedge\omega^{-1}(X))=0$ for all $X\in\mathfrak{g}$.

Similarly, for $Z\in\mathfrak{z}_\mathfrak{g}(\Hol_\mathscr{e}(\mathscr{G},\omega))$, we know that $\omega^{-1}(X)\omega(\widetilde{Z})=-\ad_X\omega(\widetilde{Z})$, so $\Omega(\widetilde{Z}\wedge\omega^{-1}(X))+\omega([\widetilde{Z},\omega^{-1}(X)])=0$, hence $[\widetilde{Z},\omega^{-1}(X)]=0$ if and only if $\Omega(\widetilde{Z}\wedge\omega^{-1}(X))=0$. Because $\Rt{h*}\widetilde{Z}=\widetilde{Z}$ for all $h\in H$, this means that $\widetilde{Z}$ is an infinitesimal automorphism if and only if it is a complete vector field and $\Omega(\widetilde{Z}\wedge\omega^{-1}(X))=0$ for all $X\in\mathfrak{g}$.

Thinking of right-invariant vector fields as parallel sections of the adjoint tractor connection, this equivalence of infinitesimal automorphisms and right-invariant vector fields modulo curvature is essentially a consequence of Lemma 1.5.12 in \cite{CapSlovak}.

We can also constrain how far $\omega(\xi_\mathscr{g})$ can be from $\Ad_{\mathscr{g}^{-1}}\omega(\xi_\mathscr{e})$ in terms of the holonomy.

\begin{proposition}Suppose $\xi\in\mathfrak{aut}(\mathscr{G},\omega)$. Let $\mathfrak{hol}_\mathscr{e}(\mathscr{G},\omega)$ denote the Lie subalgebra of $\mathfrak{g}$ whose one-parameter subgroups are contained in $\Hol_\mathscr{e}(\mathscr{G},\omega)$. Then, if $g$ is a development of $\mathscr{g}$ from $\mathscr{e}$, then \[\omega(\xi_\mathscr{g})\in\Ad_{g^{-1}}\left(\omega(\xi_\mathscr{e})+\mathfrak{hol}_\mathscr{e}(\mathscr{G},\omega)\right).\]\end{proposition}
\begin{proof}Let $\gamma:\mathbb{R}\to\mathscr{G}$ be the curve $t\mapsto\exp(t\xi)\mathscr{g}$. We want to show that $\dot{\gamma}_{{}_G}(0)=\omega(\dot{\gamma}(0))\in\Ad_{g^{-1}}(\omega(\xi_\mathscr{e})+\mathfrak{hol}_\mathscr{e}(\mathscr{G},\omega))$. Because $\xi$ is an infinitesimal automorphism, we know that $g\gamma_{{}_G}(t)\in\Hol_\mathscr{e}(\mathscr{G},\omega)\exp(t\omega(\xi_\mathscr{e}))g$, so we get a curve $\gamma_{\Hol}:t\mapsto g\gamma_{{}_G}(t)g^{-1}\exp(t\omega(\xi_\mathscr{e}))^{-1}\in\Hol_\mathscr{e}(\mathscr{G},\omega)$ such that $\gamma_{\Hol}(0)=e$ and $\gamma_{{}_G}(t)=g^{-1}(\gamma_{\Hol}(t)\exp(t\omega(\xi_\mathscr{e})))g$, hence \[\dot{\gamma}_{{}_G}(0)=\Ad_{g^{-1}}(\dot{\gamma}_{\Hol}(0)+\omega(\xi_\mathscr{e}))\in\Ad_{g^{-1}}\left(\mathfrak{hol}_\mathscr{e}(\mathscr{G},\omega)+\omega(\xi_\mathscr{e})\right).\quad\mbox{\qedhere}\]\end{proof}

\section{Applications}\label{applications}
\subsection{Product decompositions}
We can use the ideas above to generalize the de Rham decomposition theorem for Riemannian manifolds. To get there, we will need a couple of definitions.

\begin{definition}Suppose $(G,H)$ and $(Q,K)$ are model geometries. Let $i:H\to K$ be a homomorphism and $\lambda:\mathfrak{g}\to\mathfrak{q}$ be a linear map such that $\lambda\circ\Ad_h=\Ad_{i(h)}\circ\lambda$ for all $h\in H$, $\lambda|_\mathfrak{h}=i_*$, and $\bar{\lambda}:\mathfrak{g}/\mathfrak{h}\to\mathfrak{q}/\mathfrak{k}$ is a linear isomorphism. Then, the \emph{extension functor} corresponding to $(i,\lambda)$ takes Cartan geometries of type $(G,H)$ to Cartan geometries of type $(Q,K)$ by \[(\mathscr{G},\omega)\longmapsto(\mathscr{G}\times_iK,\omega_\lambda),\] where $\omega_\lambda$ is uniquely characterized by $j^*\omega_\lambda=\lambda(\omega)$, with $j:\mathscr{g}\mapsto(\mathscr{g},e)$.\end{definition}

Using inclusion maps for both $i$ and $\lambda$ gives a way to reverse a holonomy reduction on a single curved orbit. That is, if $(\mathscr{G},\omega)$ is a Cartan geometry of type $(G,H)$, and we make a holonomy reduction with basepoint $\mathscr{e}\in\mathscr{G}$ and subgroup $J$ such that $q_{{}_H}(J\bullet\mathscr{e})=q_{{}_H}(\mathscr{G})$, then $J\bullet\mathscr{e}$ has the structure of a Cartan geometry of type $(J,H\cap J)$, and the extension functor using inclusion maps recovers $(\mathscr{G},\omega)$ from $J\bullet\mathscr{e}$.

We can, then, define product geometries as extensions of Cartesian products of two given geometries.

\begin{definition}Suppose $(A,P)$, $(B,K)$, and $(G,H)$ are model geometries such that $i:A\times B\hookrightarrow G$ determines an extension functor from $(A\times B,P\times K)$ to $(G,H)$. Let $(\mathscr{A},\alpha)$ be a Cartan geometry of type $(A,P)$ and $(\mathscr{B},\beta)$ be a Cartan geometry of type $(B,K)$. Then, a Cartan geometry $(\mathscr{G},\omega)$ is the \emph{product of type $(G,H)$} of $(\mathscr{A},\alpha)$ and $(\mathscr{B},\beta)$ if and only if it is geometrically isomorphic to the extension of $(\mathscr{A}\times\mathscr{B},\alpha+\beta)$ to type $(G,H)$.\end{definition}

Equivalently, we could say that $(\mathscr{G},\omega)$ is the product of type $(G,H)$ of $(\mathscr{A},\alpha)$ and $(\mathscr{B},\beta)$ if and only if, for some $\mathscr{e}\in\mathscr{G}$, $\mathscr{e}\bullet AB$ is well-defined, $q_{{}_H}(\mathscr{e}\bullet AB)=q_{{}_H}(\mathscr{G})$, and the induced geometric structure on $\mathscr{e}\bullet AB$ is geometrically isomorphic to $(\mathscr{A}\times\mathscr{B},\alpha+\beta)$.

For example, the models $(\Iso(m),\Orth(m))$ and $(\Iso(n),\Orth(n))$ for Riemannian geometries in dimension $m$ and $n$, respectively, fit together, so that the inclusion map $\Iso(m)\times\Iso(n)\hookrightarrow\Iso(m+n)$ gives an extension functor from $(\Iso(m)\times\Iso(n),\Orth(m)\times\Orth(n))$ to $(\Iso(m+n),\Orth(m+n))$. Thus, given two Riemannian geometries, we can take their product and extend to a new Riemannian geometry, which is just the standard Riemannian product.

By vague analogy to how a connected Lie group $G$ with closed subgroups $A$ and $B$ such that $\mathfrak{a}+\mathfrak{b}=\mathfrak{g}$ is isomorphic to the product $A\times B$ if and only if $A\cap B=[A,B]=\{e\}$, the following theorem gives necessary and sufficient conditions for a Cartan geometry to be a product.


\begin{theorem}\label{productstructure} Suppose $(A,P)$, $(B,K)$, and $(G,H)$ are model geometries such that $A\times B\hookrightarrow G$ determines an extension functor from $(A\times B,P\times K)$ to $(G,H)$. Let $(\mathscr{G},\omega)$ be a Cartan geometry of type $(G,H)$ over a connected smooth manifold $M$, with quotient map $q_{{}_H}:\mathscr{G}\to M$, and fix $\mathscr{e}\in\mathscr{G}$ such that $\mathscr{e}\circ A$, $\mathscr{e}\circ B$, and $\mathscr{e}\circ AB$ are well-defined and $q_{{}_H}(\mathscr{e}\circ AB)=M$. Then, $(\mathscr{G},\omega)$ is the product of type $(G,H)$ of $\mathscr{e}\circ A$ and $\mathscr{e}\circ B$ if and only if the following hold:
\begin{itemize}
\item $\mathscr{e}\circ A\cap\mathscr{e}\circ B=\{\mathscr{e}\}$;
\item $\exp(\omega^{-1}(X))\exp(\omega^{-1}(Y))\mathscr{g}=\exp(\omega^{-1}(Y))\exp(\omega^{-1}(X))\mathscr{g}$ for every $X\in\mathfrak{a}$, $Y\in\mathfrak{b}$, and $\mathscr{g}\in\mathscr{e}\circ AB$ such that either side is well-defined; 
\item if $\mathscr{g}\in\mathscr{e}\circ AB$ and $\mathscr{g}h\in\mathscr{e}\circ AB$ for some $h\in H$, then $h\in PK$.
\end{itemize}\end{theorem}
\begin{proof}The forward implication is clear, so we just need to prove that the conditions are sufficient to give the product structure. To do this, we will define a geometric map that is clearly an isomorphism if well-defined, then prove that it is well-defined. For the purpose of clean presentation, we will use the notation $\exp(\omega^{-1}(\overline{X}))$ to denote a composition $\exp(\omega^{-1}(X_k))\cdots\exp(\omega^{-1}(X_1))$ throughout this proof.

Suppose the conditions given above are satisfied. Since $q_{{}_H}(\mathscr{e}\circ AB)=M$, every $\mathscr{g}\in\mathscr{G}$ is of the form $\mathscr{g}'h$ for some $\mathscr{g}'\in\mathscr{e}\circ AB$ and $h\in H$, and by the condition on the flows, every $\mathscr{g}'\in\mathscr{e}\circ AB$ is of the form $(\exp(\omega^{-1}(\overline{X}))\exp(\omega^{-1}(\overline{Y}))\mathscr{e})pk$, where $p\in P$, $k\in K$, and each $X_i\in\mathfrak{a}$ and $Y_i\in\mathfrak{b}$. Therefore, every $\mathscr{g}\in\mathscr{G}$ is of the form $(\exp(\omega^{-1}(\overline{X}))\exp(\omega^{-1}(\overline{Y}))\mathscr{e})h$ for some $X_i\in\mathfrak{a}$, $Y_i\in\mathfrak{b}$, and $h\in H$, so we can consider a map $\varphi:\mathscr{G}\to(\mathscr{e}\circ A\times\mathscr{e}\circ B)\times_{PK}H$ given by \[(\exp(\omega^{-1}(\overline{X}))\exp(\omega^{-1}(\overline{Y}))\mathscr{e})h\mapsto(\exp(\omega^{-1}(\overline{X}))\mathscr{e},\exp(\omega^{-1}(\overline{Y}))\mathscr{e},h).\]

Let us suppose that $\varphi$ is well-defined. Then, it is $H$-equivariant, and since the Cartan connections on the geometric leaves are given by taking the Cartan connection on $\mathscr{G}$ and pulling back by inclusion, it follows that $\varphi$ is geometric. Moreover, since $(\mathscr{a},\mathscr{b},h)=(\mathscr{a}',\mathscr{b}',h')$ implies that there are $p\in P$ and $k\in K$ such that $pkh=h'$, $\mathscr{a}p=\mathscr{a}'$, and $\mathscr{b}k=\mathscr{b}'$, $\varphi$ is injective, hence bijective. Thus, if we prove $\varphi$ is well-defined, then it will give us the desired geometric isomorphism.
 
Suppose \[(\exp(\omega^{-1}(\overline{X}))\exp(\omega^{-1}(\overline{Y}))\mathscr{e})h=(\exp(\omega^{-1}(\overline{X'}))\exp(\omega^{-1}(\overline{Y'}))\mathscr{e})h'.\] Then, \[\exp(\omega^{-1}(\overline{X}))\exp(\omega^{-1}(\overline{Y}))\mathscr{e}=(\exp(\omega^{-1}(\overline{X'}))\exp(\omega^{-1}(\overline{Y'}))\mathscr{e})h'h^{-1},\] so $h'h^{-1}\in PK$. Writing $h'h^{-1}=pk$ for $p\in P$ and $k\in K$, we get that \begin{align*}(\exp(\omega^{-1}(\overline{X}))\exp(\omega^{-1}(\overline{Y}))\mathscr{e})k^{-1} & =\exp(\omega^{-1}(\overline{X}))(\exp(\omega^{-1}(\overline{Y}))\mathscr{e})k^{-1} \\ & =(\exp(\omega^{-1}(\overline{X'}))\exp(\omega^{-1}(\overline{Y'}))\mathscr{e})p \\ & =\exp(\omega^{-1}(\overline{Y'}))(\exp(\omega^{-1}(\overline{X'}))\mathscr{e})p,\end{align*} so \[\exp(\omega^{-1}(\overline{Y'}))^{-1}\left((\exp(\omega^{-1}(\overline{Y}))\mathscr{e})k^{-1}\right)=\exp(\omega^{-1}(\overline{X}))^{-1}\left((\exp(\omega^{-1}(\overline{X'}))\mathscr{e})p\right).\] Thus, since \[\exp(\omega^{-1}(\overline{Y'}))^{-1}\left((\exp(\omega^{-1}(\overline{Y}))\mathscr{e})k^{-1}\right)\in\mathscr{e}\circ B\] and \[\exp(\omega^{-1}(\overline{X}))^{-1}\left((\exp(\omega^{-1}(\overline{X'}))\mathscr{e})p\right)\in\mathscr{e}\circ A,\] both are equal to $\mathscr{e}$ because we are assuming $\mathscr{e}\circ A\cap\mathscr{e}\circ B=\{\mathscr{e}\}$. It follows that $\exp(\omega^{-1}(\overline{X}))\mathscr{e}=(\exp(\omega^{-1}(\overline{X'}))\mathscr{e})p$ and $\exp(\omega^{-1}(\overline{Y}))\mathscr{e}=(\exp(\omega^{-1}(\overline{Y'}))\mathscr{e})k$, hence \begin{align*}\varphi\left((\exp(\omega^{-1}(\overline{X}))\exp(\omega^{-1}(\overline{Y}))\mathscr{e})h\right) & =(\exp(\omega^{-1}(\overline{X}))\mathscr{e},\exp(\omega^{-1}(\overline{Y}))\mathscr{e},h) \\ & =((\exp(\omega^{-1}(\overline{X'}))\mathscr{e})p,(\exp(\omega^{-1}(\overline{Y'}))\mathscr{e})k,h) \\ & =(\exp(\omega^{-1}(\overline{X'}))\mathscr{e},\exp(\omega^{-1}(\overline{Y'}))\mathscr{e},pkh) \\ & =(\exp(\omega^{-1}(\overline{X'}))\mathscr{e},\exp(\omega^{-1}(\overline{Y'}))\mathscr{e},h') \\ & =\varphi\left((\exp(\omega^{-1}(\overline{X'}))\exp(\omega^{-1}(\overline{Y'}))\mathscr{e})h'\right)\!.~\mbox{\qedhere}\end{align*}\end{proof}

Of the above three conditions, the one that can be conceivably strenuous to verify is that the flows commute whenever either composition makes sense. As such, additional constraints on a geometry that can simplify this condition can be helpful in practice. Toward this end, note that there are two parts to verifying that the flow condition holds: that $\exp(\omega^{-1}(X))\exp(\omega^{-1}(Y))$ and $\exp(\omega^{-1}(Y))\exp(\omega^{-1}(X))$ are equal on $\mathscr{e}\circ AB$ wherever both are defined and that if either of them is well-defined at some point of $\mathscr{e}\circ AB$, then so is the other.

The first part is fairly straightforward to get around: given two vector fields $V$ and $W$, $\exp(sV)\exp(tW)=\exp(tW)\exp(sV)$ for all $s,t\in\mathbb{R}$ such that both sides are well-defined if and only if $[V,W]=0$. In our case, $[\omega^{-1}(X),\omega^{-1}(Y)]=\omega^{-1}([X,Y]-\Omega^\omega(X\wedge Y))$, so since $[X,Y]=0$ for $X\in\mathfrak{a}$ and $Y\in\mathfrak{b}$, the first part of the flow condition holds if and only if $\Omega^\omega_\mathscr{g}(X\wedge Y)=0$ for all $X\in\mathfrak{a}$, $Y\in\mathfrak{b}$, and $\mathscr{g}\in\mathscr{e}\circ AB$.

This condition on the curvature is sufficient to apply Theorem \ref{productstructure} locally. In other words, if $\mathscr{e}\circ AB$ is well-defined and $\Omega^{\omega}_\mathscr{g}(X\wedge Y)=0$ for all $X\in\mathfrak{a}$, $Y\in\mathfrak{b}$, and $\mathscr{g}\in\mathscr{e}\circ AB$, then for each $\mathscr{g}\in\mathscr{e}\circ AB$ such that $\mathscr{g}\circ A$ and $\mathscr{g}\circ B$ are well-defined, there exists a neighborhood of $U$ of $q_{{}_H}(\mathscr{g})\in M$ such that the restricted geometry $(q_{{}_H}^{-1}(U),\omega)$ satisfies the conditions of the theorem when we take the basepoint to be $\mathscr{g}$. We can get such a $U$ by taking sufficiently small open sets $U_\mathfrak{a}\subset\mathfrak{a}$ and $U_\mathfrak{b}\subset\mathfrak{b}$ and setting $U=q_{{}_H}(\exp(\omega^{-1}(U_\mathfrak{a}))\exp(\omega^{-1}(U_\mathfrak{b}))\mathscr{g})$.

The second part of the flow condition, that if either of the compositions $\exp(\omega^{-1}(X))\exp(\omega^{-1}(Y))\mathscr{g}$ or $\exp(\omega^{-1}(Y))\exp(\omega^{-1}(X))\mathscr{g}$ is well-defined for $\mathscr{g}\in\mathscr{e}\circ AB$ then so is the other, is slightly trickier. However, we can bypass it by simply guaranteeing that the flows are always well-defined. This constraint is called completeness.

\begin{definition}A Cartan geometry $(\mathscr{G},\omega)$ of type $(G,H)$ is \emph{complete} if and only if $\omega^{-1}(X)$ is a complete vector field whenever $X\in\mathfrak{g}$.\end{definition}

Thus, if $(\mathscr{G},\omega)$ is complete and $\Omega^\omega_\mathscr{g}(X\wedge Y)=0$ for all $X\in\mathfrak{a}$, $Y\in\mathfrak{b}$, and $\mathscr{g}\in\mathscr{e}\circ AB$, then the flow condition in Theorem \ref{productstructure} is automatically satisfied. If we also assume that the base manifold $M$ is simply connected, then we can further show that all three conditions are satisfied.

\begin{corollary}\label{deRham} Suppose $(A,P)$, $(B,K)$, and $(G,H)$ are model geometries such that $i:A\times B\hookrightarrow G$ determines an extension functor from $(A\times B,P\times K)$ to $(G,H)$. Let $(\mathscr{G},\omega)$ be a complete Cartan geometry of type $(G,H)$ over a simply connected smooth manifold $M$, with quotient map $q_{{}_H}:\mathscr{G}\to M$. Fix $\mathscr{e}\in\mathscr{G}$ such that $\mathscr{e}\circ A$, $\mathscr{e}\circ B$, and $\mathscr{e}\circ AB$ are well-defined and $q_{{}_H}(\mathscr{e}\circ AB)=M$. Then, $(\mathscr{G},\omega)$ is the product of type $(G,H)$ of $\mathscr{e}\circ A$ and $\mathscr{e}\circ B$ if and only if $\Omega^\omega_\mathscr{g}(X\wedge Y)=0$ whenever $X\in\mathfrak{a}$, $Y\in\mathfrak{b}$, and $\mathscr{g}\in\mathscr{e}\circ AB$.\end{corollary}
\begin{proof}As we said above, the condition on the curvature together with completeness are sufficient to guarantee the flow condition in Theorem \ref{productstructure} is satisfied. Moreover, since $q_{{}_{PK}}|_{\mathscr{e}\circ AB}:\mathscr{e}\circ AB\to (\mathscr{e}\circ AB)/PK$ is a fiber bundle map and $\bar{q}_{{}_H}:(\mathscr{e}\circ AB)/PK\to M$ is a local diffeomorphism, $q_{{}_H}|_{\mathscr{e}\circ AB}=\bar{q}_{{}_H}\circ q_{{}_{PK}}|_{\mathscr{e}\circ AB}$ is a fibration. Because $M$ is simply connected, there is a homotopy rel endpoints $\gamma$ from each loop $\gamma_0:[0,1]\to M$ based at $q_{{}_H}(\mathscr{g})$ to the constant loop $\gamma_1:t\mapsto q_{{}_H}(\mathscr{g})$, so since $q_{{}_H}|_{\mathscr{e}\circ AB}$ is a fibration, if $\widehat{\gamma_0}$ is a lift of $\gamma_0$ to $\mathscr{e}\circ AB$ with $\widehat{\gamma_0}(0)=\mathscr{g}$, then there is a lift $\widehat{\gamma}$ to $\mathscr{e}\circ AB$ of the homotopy $\gamma$ with $(\widehat{\gamma})_0=\widehat{\gamma_0}$ and $(\widehat{\gamma})_s(0)=\mathscr{g}$ for each $s\in[0,1]$. Since $\widehat{\gamma}$ is a lift of a homotopy rel endpoints and $\gamma_1$ is the constant loop, $(\widehat{\gamma})_s(1)\in(\mathscr{e}\circ AB)\cap\mathscr{g}H$ for each $s\in[0,1]$ and $(\widehat{\gamma})_1(t)\in\mathscr{g}PK$ for each $t\in[0,1]$, so $(\widehat{\gamma})_s(1)\in\mathscr{e}PK$ for each $s\in[0,1]$. In particular, if $\mathscr{g}$ and $\mathscr{g}h$ are in $\mathscr{e}\circ AB$, then $h\in PK$. Thus, it only remains to show that $\mathscr{e}\circ A\cap\mathscr{e}\circ B=\{\mathscr{e}\}$.

Suppose $\mathscr{g}\in\mathscr{e}\circ A\cap\mathscr{e}\circ B$. Then, there exist paths $\alpha:[0,1]\to\mathscr{e}\circ A$ and $\beta:[0,1]\to\mathscr{e}\circ B$, $p\in P$, and $k\in K$ such that $\alpha(0)=\beta(0)=\mathscr{e}$ and $\alpha(1)p=\beta(1)k=\mathscr{g}$. Again using that $q_{{}_H}|_{\mathscr{e}\circ AB}$ is a fibration and $M$ is simply connected, there exists a homotopy $\widehat{\gamma}:[0,1]^2\to\mathscr{e}\circ AB$ such that $\widehat{\gamma}_0=\alpha$ and $\widehat{\gamma}_1=\beta\zeta$, where $\zeta:[0,1]\to PK$. Moreover, by right-translating by a path in $PK$ parametrized by $s$ if necessary, we may assume that $\widehat{\gamma}_s(0)=\mathscr{e}$ for all $s\in[0,1]$.

Let $\pi:A\times B\to A$ be the natural quotient map. For $\sigma:[0,1]\to\mathscr{e}\circ AB$ with $\sigma(0)=\mathscr{e}$, define $\pi_\mathscr{e}(\sigma):[0,1]\to\mathscr{e}\circ A$ to be the unique path such that $\pi_\mathscr{e}(\sigma)(0)=\mathscr{e}$ and $\pi_\mathscr{e}(\sigma)_G=\pi(\sigma_G)$.

\begin{lemma}Suppose $\widehat{\gamma}:[0,1]^2\to\mathscr{e}\circ AB$ is a homotopy between two paths $\widehat{\gamma}_0$ and $\widehat{\gamma}_1$ such that, for every $s\in[0,1]$, $\widehat{\gamma}_s(0)=\mathscr{e}$ and $\widehat{\gamma}_s(1)=\widehat{\gamma}_0(1)h_s$ for some $h_s\in PK$. Then, $\pi_\mathscr{e}(\widehat{\gamma}_s)(1)=\pi_\mathscr{e}(\widehat{\gamma}_0)(1)\pi(h_s)$.\end{lemma}
\begin{proof}We will restrict to the case where, for all $s\in[0,1]$, $q_{{}_H}(\widehat{\gamma}_s(t))=q_{{}_H}(\widehat{\gamma}_0(t))$ for all $t$ outside of a small interval $[t_1,t_2]$ for which $q_{{}_H}(\widehat{\gamma}([0,1]\times[t_1,t_2]))$ is contained in an open set $U\subseteq M$ such that, for some $\mathscr{g}'\in q_{{}_H}^{-1}(U)\cap\mathscr{e}\circ AB$, there is a geometric isomorphism \[\varphi:q_{{}_H}^{-1}(U)\cap\mathscr{e}\circ AB\to(\mathscr{g}'\circ A)_{q_{{}_H}^{-1}(U)}\times(\mathscr{g}'\circ B)_{q_{{}_H}^{-1}(U)},\] where $(\mathscr{g}'\circ A)_{q_{{}_H}^{-1}(U)}$ and $(\mathscr{g}'\circ B)_{q_{{}_H}^{-1}(U)}$ are the geometric leaves for $A$ and $B$ through $\mathscr{g}'$ taken in the restricted geometry $(q_{{}_H}^{-1}(U),\omega)$, as opposed to in $(\mathscr{G},\omega)$ where, a priori, $\mathscr{g}'\circ A$ might intersect $\mathscr{g}'\circ B$ at a point other than $\mathscr{g}'$. Note that, because $\Omega^\omega_\mathscr{g}(X\wedge Y)=0$ for all $\mathscr{g}\in\mathscr{e}\circ AB$ and $(\mathscr{G},\omega)$ is complete, if \[\mathscr{g}'=(\exp(\omega^{-1}(X_1))\cdots\exp(\omega^{-1}(X_\ell))\exp(\omega^{-1}(Y_1))\cdots\exp(\omega^{-1}(Y_{\ell'}))\mathscr{e})pk\] for $X_i\in\mathfrak{a}$, $Y_i\in\mathfrak{b}$, $p\in P$, and $k\in K$, then \[R_k\circ\exp(\omega^{-1}(Y_1))\cdots\exp(\omega^{-1}(Y_{\ell'}))\] gives a geometric isomorphism between $\mathscr{e}\circ A$ and $\mathscr{g}'\circ A$, so in particular $\mathscr{g}'\circ A$ is well-defined, and similarly for $\mathscr{g}'\circ B$. Moreover, as was noted above, the curvature condition guarantees that each $\mathscr{g}'\in\mathscr{e}\circ AB$ has an open set $U$ in $M$ such that $(q_{{}_H}^{-1}(U),\omega)$ splits as a product of type $(G,H)$ of the restricted geometric leaves if they are well-defined. Thus, since $[0,1]^2$ is compact, we may always piece together a homotopy of the type described in the lemma from finitely many of these homotopies where the projected curves on $M$ only differ in an open set $U$ of the desired type, so no generality is lost in making this restriction. This is essentially the same trick used in Lemma 6 of Section IV.6 in \cite{KobayashiNomizu}.

By assumption, for all $t\not\in[t_1,t_2]$, $\widehat{\gamma}_s(t)=\widehat{\gamma}_0(t)\sigma_s(t)$ for some $\sigma:[0,1]\times([0,t_1)\cup(t_2,1])\to PK$. For $t\in[0,t_1)$, this means $(\widehat{\gamma}_s)_G(t)=(\widehat{\gamma}_0)_G(t)\sigma_s(t)$, so $\pi_\mathscr{e}(\widehat{\gamma}_s)(t)=\pi_\mathscr{e}(\widehat{\gamma}_0)(t)\pi(\sigma_s(t))$.

Let $t_0,t_3\in[0,1]$ such that $t_0<t_1<t_2<t_3$ and $q_{{}_H}(\widehat{\gamma}([0,1]\times[t_0,t_3]))\subseteq U$. Denoting by $\mathrm{pr}_A$ and $\mathrm{pr}_B$ the natural projections from $(\mathscr{g}'\circ A)_{q_{{}_H}^{-1}(U)}\times(\mathscr{g}'\circ B)_{q_{{}_H}^{-1}(U)}$ to $(\mathscr{g}'\circ A)_{q_{{}_H}^{-1}(U)}$ and $(\mathscr{g}'\circ B)_{q_{{}_H}^{-1}(U)}$, respectively, and by $\widehat{\gamma}_s(\mathrm{id}+t_0)$ the path $\widehat{\gamma}_s(\mathrm{id}+t_0):[0,t_3-t_0]\to q_{{}_H}^{-1}(U)\cap\mathscr{e}\circ AB$ with $\widehat{\gamma}_s(\mathrm{id}+t_0)(t)=\widehat{\gamma}_s(t+t_0)$, we get \[\widehat{\gamma}_s(\mathrm{id}+t_0)_G=\mathrm{pr}_A(\varphi(\widehat{\gamma}_s(\mathrm{id}+t_0)))_G\mathrm{pr}_B(\varphi(\widehat{\gamma}_s(\mathrm{id}+t_0)))_G.\] By composing flows of vector fields $\omega^{-1}(Y)$ for $Y\in\mathfrak{b}$ and right-translations $R_k$ with $k\in K$, we can get a geometric isomorphism $\lambda:\mathscr{g}'\circ A\to\mathscr{e}\circ A$ such that $\lambda(\mathrm{pr}_A(\varphi(\widehat{\gamma}_0(t_0))))=\pi_\mathscr{e}(\widehat{\gamma}_0)(t_0)$. In particular, \begin{align*}\lambda(\mathrm{pr}_A(\varphi(\widehat{\gamma}_s(t_0)))) & =\lambda(\mathrm{pr}_A(\varphi(\widehat{\gamma}_0(t_0)\sigma_s(t_0))))=\lambda(\mathrm{pr}_A(\varphi(\widehat{\gamma}_0(t_0))))\pi(\sigma_s(t_0)) \\ & =\pi_\mathscr{e}(\widehat{\gamma}_0)(t_0)\pi(\sigma_s(t_0))=\pi_\mathscr{e}(\widehat{\gamma}_s)(t_0).\end{align*} Thus, $\pi_\mathscr{e}(\widehat{\gamma}_s)(t)=\lambda(\mathrm{pr}_A(\varphi(\widehat{\gamma}_s(t))))$ for all $t\in[t_0,t_3]$, since $\lambda(\mathrm{pr}_A(\varphi(\widehat{\gamma}_s(\mathrm{id}+t_0))))_G=\pi(\widehat{\gamma}_s(\mathrm{id}+t_0)_G)$. Moreover, since $t_3\in(t_2,1]$, we see that $\pi_\mathscr{e}(\widehat{\gamma}_s)(t_3)=\pi_\mathscr{e}(\widehat{\gamma}_0)(t_3)\pi(\sigma_s(t_3))$.

Finally, for $t\in(t_2,1]$, we still have $\widehat{\gamma}_s(t)=\widehat{\gamma}_0(t)\sigma_s(t)$, so we still get $\pi_\mathscr{e}(\widehat{\gamma}_s)(t)=\pi_\mathscr{e}(\widehat{\gamma}_0)(t)\pi(\sigma_s(t))$, hence \[\pi_\mathscr{e}(\widehat{\gamma}_s)(1)=\pi_\mathscr{e}(\widehat{\gamma}_0)(1)\pi(\sigma_s(1))=\pi_\mathscr{e}(\widehat{\gamma}_0)(1)\pi(h_s).\quad\mbox{\qedhere}\]\end{proof}

For our homotopy $\widehat{\gamma}$ from $\alpha$ to $\beta\zeta$, we have that $\pi_\mathscr{e}(\widehat{\gamma}_0)=\alpha$ and $\pi_\mathscr{e}(\widehat{\gamma}_1)=\mathscr{e}\pi(\zeta)$, so by the lemma, \[\mathscr{e}\pi(\zeta(1))=\alpha(1)\pi(pk^{-1}\zeta(1))=\alpha(1)p\pi(\zeta(1))=\mathscr{g}\pi(\zeta(1)),\] hence $\mathscr{g}=\mathscr{e}$.\mbox{\qedhere}\end{proof}

Corollary \ref{deRham} generalizes the de Rham decomposition for Riemannian manifolds. Unlike in the usual de Rham decomposition, we cannot always get a product decomposition of the geometry from a product decomposition of a subgroup containing the holonomy group when the geometry is complete and the base manifold is simply connected. Indeed, by identifying $\mathfrak{g}$ with $\mathbb{R}^{\dim(\mathfrak{g})}$, the Maurer-Cartan form on a Lie group $G$ gives a complete Cartan connection of type $(\mathbb{R}^{\dim(\mathfrak{g})},\{0\})$, but not all simply connected Lie groups decompose into a product of 1-dimensional submanifolds. Thus, in this more general case, we need to use the additional finesse of the curvature, rather than just the holonomy group.

We should remark that an alternative proof of Corollary \ref{deRham} can be obtained by using the generalized Cartan-Ambrose-Hicks theorem from \cite{BlumenthalHebda1989}. This would be similar to the approach from \cite{Wu}, except it would apply to this more general case.

For reductive\footnote{We say that a Cartan geometry $(\mathscr{G},\omega)$ of type $(G,H)$ is \emph{reductive} if and only if there exists an $\Ad_H$-invariant subspace $\mathfrak{m}\subseteq\mathfrak{g}$ such that $\mathfrak{g}\approx\mathfrak{m}\oplus\mathfrak{h}$ as an $\Ad_H$-module.} Cartan geometries, such as first-order $H$-structures with a choice of connection, completeness is not a particularly strong constraint. Indeed, it is well-known that reductive Cartan geometries are complete if and only if they are geodesically complete with respect to the induced covariant derivative. However, it should be noted that, in general, the completeness hypothesis of Corollary \ref{deRham} is an exceptionally fanciful assumption. Completeness remains rather mysterious for most Cartan geometries, but it is known in many cases to be quite restrictive. For example, it was proved in \cite{McKay} that every complete complex parabolic geometry is necessarily flat. On the other hand, we only needed to use that $\mathscr{e}\circ AB$ was complete in the proof of Corollary \ref{deRham}, and there are many Cartan geometries that have holonomy reductions to complete reductive geometries, so the result still applies quite generally.

Moreover, we can still get a local decomposition without completeness, and this can lead to interesting results---even for parabolic geometries---when paired with an appropriate holonomy reduction. For example, it is not too difficult to imagine how useful results like Theorem 4.5 of \cite{Armstrong} might be gleaned from Theorem \ref{productstructure} above.

\subsection{Inessential automorphism groups of parabolic geometries}
For this section, consider the following two trivial consequences of Corollary \ref{essential}.

\begin{proposition}\label{subset} Suppose $(\mathscr{G},\omega)$ is a Cartan geometry of type $(G,H)$ over $M$ with a holonomy reduction to $J$ based at $\mathscr{e}\in\mathscr{G}$ such that $q_{{}_H}(J\bullet\mathscr{e})=q_{{}_H}(\mathscr{G})=M$. Then, for arbitrary $\mathscr{g}\in\mathscr{G}$, the action of $K\leq\Aut(\mathscr{G},\omega)$ on $\mathscr{G}$ restricts to an action by automorphisms on $J\bullet\mathscr{g}$ if and only if $K\mathscr{e}=\{\varphi(\mathscr{e}):\varphi\in K\}\subseteq\mathscr{e}\bullet J$.\end{proposition}

\begin{proposition}\label{fixedpoint} Suppose $(\mathscr{G},\omega)$ is a Cartan geometry of type $(G,H)$ over $M$ with a holonomy reduction to $J$ based at $\mathscr{e}\in\mathscr{G}$ such that $q_{{}_H}(J\bullet\mathscr{e})=q_{{}_H}(\mathscr{G})=M$. Then, if $\xi$ is an infinitesimal automorphism of $(\mathscr{G},\omega)$ such that $q_{{}_H*}(\xi)$ vanishes at some $x\in M$, then $\xi$ restricts to an infinitesimal automorphism of $J\bullet\mathscr{g}$ for some $\mathscr{g}\in q_{{}_H}^{-1}(x)$ if and only if $\omega_{\mathscr{g}}(\xi)\in\Ad_{\mathscr{g}^{-1}}(\mathfrak{j})\cap\mathfrak{h}$.\end{proposition}

Since $q_{{}_H}(J\bullet\mathscr{e})=q_{{}_H}(\mathscr{G})$ in the above proposition, every $\mathscr{g}\in\mathscr{G}$ has a development of the form $g=jh$ for some $j\in J$ and $h\in H$, in which case we get $\Ad_{g^{-1}}(\mathfrak{j})\cap\mathfrak{h}=\Ad_{(jh)^{-1}}(\mathfrak{j})\cap\mathfrak{h}=\Ad_{h^{-1}}(\mathfrak{j}\cap\mathfrak{h})$. In particular, we can always choose $\mathscr{g}\in J\bullet\mathscr{e}$ and $\omega_\mathscr{g}(\xi)\in\mathfrak{j}\cap\mathfrak{h}$ in the situation of Proposition \ref{fixedpoint}.

Both of these propositions allow us to use a relatively small amount of information to determine global properties as long as we know the holonomy group. In the first case, it tells us that a subgroup of automorphisms of a geometric structure restricts to a subgroup of automorphisms of an underlying structure if and only if the subgroup's orbit from the chosen basepoint for the holonomy is contained in the curved coset corresponding to the holonomy reduction for the underlying structure. In the second case, we can say whether an infinitesimal automorphism of a geometric structure with a fixed point on the base manifold restricts to an automorphism of an underlying structure by just checking its value at a single point in the fiber above the fixed point.

These results would be particularly useful for exploring inessential automorphisms of parabolic geometries, which are automorphisms of an infinitesimal flag structure that restrict to automorphisms of a particular underlying geometry. Indeed, a local analogue of Proposition \ref{fixedpoint}, not relying on holonomy, has appeared before in \cite{Alt}, generalizing results from \cite{Frances2}, for precisely this purpose. However, there are obstructions that may prevent the above propositions from giving us global information about the underlying infinitesimal flag structures, which are often the focus of research for parabolic geometries. We will discuss these below. For an introduction to parabolic geometries, we highly recommend \cite{CapSlovak}; for brevity, we will assume some familiarity with the subject for the rest of this section.

Given a model parabolic geometry $(G,P)$, we get an associated filtration $\mathfrak{g}=\mathfrak{g}^{-k}\supset\cdots\supset\mathfrak{g}^0=\mathfrak{p}\supset\cdots\supset\mathfrak{g}^k\supset\{0\}$ and grading $\mathfrak{g}=\mathfrak{g}_{-k}+\cdots+\mathfrak{g}_0+\cdots+\mathfrak{g}_k$. We will denote by $\mathfrak{p}_+$ the nilradical $\mathfrak{g}_1+\cdots+\mathfrak{g}_k$ of $\mathfrak{p}$ and by $\mathfrak{g}_-$ the nilradical $\mathfrak{g}_{-k}+\cdots+\mathfrak{g}_{-1}$ of the opposite parabolic of $\mathfrak{p}$. We also get closed subgroups $G_0$, $G_-$, and $P_+$ with corresponding Lie subalgebras $\mathfrak{g}_0$, $\mathfrak{g}_-$, and $\mathfrak{p}_+$, respectively, and $P\simeq G_0\ltimes P_+$.

The subgroup $G_0$ is reductive, so it splits as $\mathfrak{z}(\mathfrak{g}_0)\oplus\mathfrak{g}_0^\text{ss}$, where $\mathfrak{g}_0^\text{ss}$ is semisimple. If $\kgf$ is the Killing form on $\mathfrak{g}$, then the restriction of $\kgf$ to $\mathfrak{g}_0$ will still be nondegenerate and respect this splitting. Thus, given a homomorphism $\lambda:G_0\to\mathbb{R}_+$, there will be a unique element $\lambda_*^\kgf\in\mathfrak{z}(\mathfrak{g}_0)$ such that the homomorphism $\lambda_*:\mathfrak{g}_0\to\mathbb{R}$ satisfies $\lambda_*(Z)=\kgf(\lambda_*^\kgf,Z)$ for all $Z\in\mathfrak{g}_0$. Following \cite{CapSlovak2}, we call $\lambda_*^\kgf$ a \emph{scaling element} if and only if the restriction of $\ad_{\lambda_*^\kgf}$ to each $G_0$-irreducible component of $\mathfrak{p}_+$ is given by multiplication by a nonzero real scalar. We imagine that this was the intended definition in \cite{Alt}, since all of the results from that paper seem to work otherwise verbatim with this definition, but as written it was assumed that $\lambda_*^\kgf$ acts by scalar multiplication on each grading component $\mathfrak{g}_i$, in which case it is not hard to see that $\lambda_*^\kgf$ must be a scalar multiple of the grading element.

Recall that a parabolic geometry of type $(G,P)$ naturally induces an infinitesimal flag structure of type $(G,P)$, and that we can identify this infinitesimal flag structure with the equivalence class of parabolic geometries inducing it.

\begin{definition}Given an infinitesimal flag structure of type $(G,P)$ and a homomorphism $\lambda:G_0\to\mathbb{R}_+$ with $\lambda_*^\kgf$ a scaling element, we say a \emph{subordinate $\lambda$-exact geometry} is a choice of holonomy reduction to $G_-\ker(\lambda)$ on some $(\mathscr{G},\omega)$ of type $(G,P)$ inducing that infinitesimal flag structure such that $q_{{}_P}(G_-\ker(\lambda)\bullet\mathscr{e})=q_{{}_P}(\mathscr{G})$.\end{definition}

\begin{lemma}\label{exactWeyl} Suppose $(\mathscr{G},\omega)$ is a Cartan geometry of type $(G,P)$ over $M$ inducing a given infinitesimal flag structure and $\lambda:G_0\to\mathbb{R}_+$ is a homomorphism with $\lambda_*^\kgf$ a scaling element. If $\mathscr{e}\in\mathscr{G}$ is such that $\Hol_\mathscr{e}(\mathscr{G},\omega)\leq G_-\ker(\lambda)$ and $q_{{}_P}(G_-\ker(\lambda)\bullet\mathscr{e})=M$, then there is an exact (with respect to $\lambda$) Weyl structure $\sigma:\mathscr{G}/P_+\to\mathscr{G}$ induced by the subordinate $\lambda$-exact geometry $G_-\ker(\lambda)\bullet\mathscr{e}$. Conversely, every exact Weyl structure for $(\mathscr{G},\omega)$ induces a subordinate $\lambda$-exact geometry on some Cartan geometry $(\mathscr{G},\omega')$ inducing the same infinitesimal flag structure as $(\mathscr{G},\omega)$.\end{lemma}
\begin{proof}For a Weyl structure $\sigma:\mathscr{G}/P_+\to\mathscr{G}$, recall that $\sigma^*\omega$ is a $G_0$-equivariant $\mathfrak{g}$-valued 1-form on $\mathscr{G}/P_+$, so we may define $G_0$-equivariant $\mathfrak{g}_i$-valued 1-forms $\sigma^*\omega_i$ for each graded component $\mathfrak{g}_i$. From this, we get a Cartan geometry $(\mathscr{G}/P_+,\upsilon)$ of type $(G_-G_0,G_0)$ over $M$, where $\upsilon=\sigma^*\omega_{-k}+\cdots+\sigma^*\omega_0$. Using the extension functor from $(G_-G_0,G_0)$ to $(G,P)$ given by inclusion maps, we also get a Cartan geometry $(\mathscr{G},\omega')$ of type $(G,P)$ from $(\mathscr{G}/P_+,\upsilon)$, where $\omega'$ is uniquely characterized by $\sigma^*\omega'=\upsilon$. Moreover, $\omega-\omega'$ takes $\omega^{-1}(\mathfrak{g}^i)$ into $\mathfrak{p}$ for each filtration component $\mathfrak{g}^i$, and it takes $\omega^{-1}(\mathfrak{p})$ to $\{0\}$, so in particular $\omega'$ induces the same infinitesimal flag structure as $\omega$.

Given a subordinate $\lambda$-exact geometry $G_-\ker(\lambda)\bullet\mathscr{e}$ in $(\mathscr{G},\omega)$, we get a corresponding Weyl structure $\sigma:\mathscr{G}/P_+\to\mathscr{G}$ by defining $\sigma(\mathscr{g}P_+)=\mathscr{g}$ for $\mathscr{g}\in G_-\ker(\lambda)\bullet\mathscr{e}$ and extending by $G_0$-equivariance. In that case, for $X\in\mathfrak{g}_-+\ker(\lambda_*)$, we see that $\lambda_*(\sigma^*\omega_0(\upsilon^{-1}(X+t\lambda_*^\kgf)))=t\lambda_*(\lambda_*^\kgf)$, so since $\lambda_*(\sigma^*\omega_0)$ descends to the induced principal $\mathbb{R}_+$-connection on $\mathscr{G}/P_+\times_\lambda\mathbb{R}_+$ and every element of $\mathscr{G}/P_+$ is uniquely of the form $\sigma^{-1}(\mathscr{g})\exp(t\lambda_*^\kgf)$ for some $\mathscr{g}\in G_-\ker(\lambda)\bullet\mathscr{e}$ and $t\in\mathbb{R}$, the principal $\mathbb{R}_+$-connection on $\mathscr{G}/P_+\times_\lambda\mathbb{R}_+$ corresponds to the one induced by the global section $\mathscr{g}P\mapsto(\mathscr{g}P_+,1)$ for $\mathscr{g}\in G_-\ker(\lambda)\bullet\mathscr{e}$, hence $\sigma$ is exact.

Conversely, suppose $\sigma:\mathscr{G}/P_+\to\mathscr{G}$ is an exact Weyl structure for $(\mathscr{G},\omega)$. Consider a global section $\sigma_0:M\to\mathscr{G}/\ker(\lambda)P_+\cong\mathscr{G}/P_+\times_\lambda\mathbb{R}_+$ that induces the same principal $\mathbb{R}_+$-connection as the one induced by $\sigma^*\omega_0$. Given a loop $\gamma:[0,1]\to M$, let $\widehat{\gamma}$ be a lift of $\sigma_0(\gamma)$ to $\mathscr{G}/P_+$. Then, $\sigma(\widehat{\gamma})$ is a lift of $\gamma$ to $\mathscr{G}$, and $\sigma(\gamma(1))\in\sigma(\gamma(0))\ker(\lambda)$, so since $\sigma(\widehat{\gamma})^*\omega'=\widehat{\gamma}^*\sigma^*\omega'=\widehat{\gamma}^*\upsilon$, and $\widehat{\gamma}^*\upsilon$ takes values in $\mathfrak{g}_-+\ker(\lambda_*)$ by construction, the holonomy based at $\mathscr{e}=\sigma(\widehat{\gamma}(0))$ of the loop $\gamma$ is contained in $G_-\ker(\lambda)$. Similarly, if $\gamma:[0,1]\to M$ is a path between two points of $M$ with $\gamma(0)=q_{{}_P}(\mathscr{e})$, then we can lift it to a path $\sigma(\widehat{\gamma})$ with $\sigma(\widehat{\gamma}(0))=\mathscr{e}$ and $\sigma(\widehat{\gamma})^*\omega'$ taking values in $\mathfrak{g}_-+\ker(\lambda_*)$, so $q_{{}_P}(\mathscr{e}\bullet G_-\ker(\lambda))=M$.\mbox{\qedhere}\end{proof}

Before defining inessential and essential automorphism groups in this setting, we need to remark upon the first of two complications: while automorphisms of a parabolic geometry necessarily induce automorphisms of the underlying infinitesimal flag structure, automorphisms of the infinitesimal flag structure do not necessarily lift to automorphisms of an arbitrary overlying parabolic geometry. Thus, the automorphisms of two parabolic geometries might be completely different even if they induce the same infinitesimal flag structure. Thankfully, if we put certain mild constraints on the curvature, namely regularity and normality, then modulo some usually negligible Lie algebra cohomological restrictions on the model geometry (see Theorem 3.1.14 in \cite{CapSlovak}), this is not an issue: automorphisms of a regular infinitesimal flag structure lift to automorphisms of the (unique) overlying normal regular parabolic geometry. Thus, in a loose sense, normality is a ``maximal symmetry'' condition on regular parabolic geometries overlying some infinitesimal flag structure.

\begin{definition}\label{defessential} Suppose $(\mathscr{G},\omega^\text{norm})$ is a regular normal parabolic geometry of type $(G,P)$ and $\lambda:G_0\to\mathbb{R}_+$ is a homomorphism with $\lambda_*^\kgf$ a scaling element. A subgroup $K\leq\Aut(\mathscr{G},\omega^\text{norm})$ is said to be \emph{$\lambda$-inessential} if and only if the action restricts to an action by automorphisms on some subordinate $\lambda$-exact geometry for the infinitesimal flag structure induced by $(\mathscr{G},\omega^\text{norm})$. Otherwise, we say that $K$ is \emph{$\lambda$-essential}. We say that $K$ is \emph{essential} if and only if it is $\lambda$-essential for each $\lambda:G_0\to\mathbb{R}_+$ with $\lambda_*^\kgf$ a scaling element.

An automorphism $\varphi\in\Aut(\mathscr{G},\omega^\text{norm})$ is called $\lambda$-essential or $\lambda$-inessential according to whether the subgroup generated by $\varphi$ is $\lambda$-essential or $\lambda$-inessential. 

Similarly, subalgebras of $\mathfrak{aut}(\mathscr{G},\omega^\text{norm})$ are called $\lambda$-essential or $\lambda$-inessential according to whether the subgroups generated by them are $\lambda$-essential or $\lambda$-inessential, and an element of $\mathfrak{aut}(\mathscr{G},\omega^\text{norm})$ is said to be $\lambda$-essential or $\lambda$-inessential according to whether the subalgebra generated by it is $\lambda$-essential or $\lambda$-inessential.\end{definition}

In light of Lemma \ref{exactWeyl}, we see that these definitions of essential and inessential are equivalent to the ones in \cite{Alt}, modulo the extended choice of scaling elements from \cite{CapSlovak2}.

Unfortunately, there is a second complication that prevents us from directly applying Propositions \ref{subset} and \ref{fixedpoint}: just like with automorphism groups, the holonomy groups of different parabolic geometries inducing the same infinitesimal flag structure are not required to be the same, and as far as the author is currently aware, there is no known method of using the holonomy groups of regular normal parabolic geometries to determine the holonomy groups of other parabolic geometries inducing the same infinitesimal flag structure. Thus, unlike in the case of automorphism groups, we cannot just use the holonomy group of the overlying regular normal parabolic geometry. Moreover, even if we get a holonomy reduction to $G_-\ker(\lambda)$ at $\mathscr{e}$, $q_{{}_P}(G_-\ker(\lambda)\bullet\mathscr{e})$ is not necessarily the whole base manifold. For most cases, this makes our definitions significantly less graceful to use than the ones from \cite{Alt}.

In lieu of elegance, a clumsy global analogue of Proposition \ref{fixedpoint} for infinitesimal flag structures would be along the lines of the following.

\begin{proposition}\label{flagaut} Suppose $(\mathscr{G},\omega^\text{norm})$ is a regular normal parabolic geometry of type $(G,P)$ over $M$, $\lambda:G_0\to\mathbb{R}_+$ is a homomorphism with $\lambda_*^\kgf$ a scaling element, and $\xi\in\mathfrak{aut}(\mathscr{G},\omega^\text{norm})$ such that $q_{{}_P*}(\xi)$ vanishes at $x\in M$. Then, $\xi$ is $\lambda$-inessential if and only if there exists an $\mathscr{e}\in q_{{}_P}^{-1}(x)$ such that $\omega^\text{norm}_\mathscr{e}(\xi)\in\ker(\lambda_*)$ and a $(\mathscr{G},\omega)$ inducing the same infinitesimal flag structure as $(\mathscr{G},\omega^\text{norm})$ on which there is a subordinate $\lambda$-exact geometry based at $\mathscr{e}$ and such that $\xi$ is also an infinitesimal automorphism of $(\mathscr{G},\omega)$.\end{proposition}

The only aspect of Proposition \ref{flagaut} that does not follow immediately from Proposition \ref{fixedpoint} and Definition \ref{defessential} is that we can use $\omega^\text{norm}(\xi)$ instead of $\omega(\xi)$, but this just follows because they differ by a horizontal 1-form.

While Proposition \ref{flagaut} does, more or less immediately, recover Theorem 1.2 from \cite{Alt}, and it does tell us that $\xi$ is $\lambda$-essential if $\omega_\mathscr{e}^\text{norm}(\xi)\not\in\ker(\lambda_*)$, the considerations above severely limit its ability to obtain more global results about $\lambda$-inessential infinitesimal automorphisms. However, if we are in a situation where we already know a subordinate $\lambda$-exact geometry, then Proposition \ref{flagaut} is much more useful. For example, when constructing a parabolic geometry, we often \textit{start} with a particular Cartan geometry $(Q,\upsilon)$ of type $(G_-\ker(\lambda),\ker(\lambda))$ and get a Cartan geometry $(\mathscr{G},\omega)$ of type $(G,P)$ by using the extension functor induced by the inclusion maps. We then get a normal Cartan geometry $(\mathscr{G},\omega^\text{norm})$ by normalizing. In this case, Proposition \ref{flagaut} tells us that $\xi\in\mathfrak{aut}(\mathscr{G},\omega)\leq\mathfrak{aut}(\mathscr{G},\omega^\text{norm})$ is $\lambda$-inessential if there is some $\mathscr{e}\in Q\subset\mathscr{G}$ such that $\omega_\mathscr{e}^\text{norm}(\xi)\in\ker(\lambda_*)$.

Similarly, we can give a clumsy version of Proposition \ref{subset} for infinitesimal flag structures.

\begin{proposition}Suppose $(\mathscr{G},\omega^\text{norm})$ is a regular normal parabolic geometry of type $(G,P)$ over $M$, $\lambda:G_0\to\mathbb{R}_+$ is a homomorphism with $\lambda_*^\kgf$ a scaling element, and $K\leq\Aut(\mathscr{G},\omega^\text{norm})$. Then, $K$ is $\lambda$-inessential if and only if there exists a $(\mathscr{G},\omega)$ inducing the same infinitesimal flag structure as $(\mathscr{G},\omega^\text{norm})$ such that $K\leq\Aut(\mathscr{G},\omega)$ and on which there is a subordinate $\lambda$-exact geometry based at $\mathscr{e}\in\mathscr{G}$ such that $K\mathscr{e}\subseteq G_-\ker(\lambda)\bullet\mathscr{e}$.\end{proposition}

Again, while this can be cumbersome, if we have foreknowledge of which parabolic geometries underlying a regular infinitesimal flag structure support a subordinate $\lambda$-exact geometry, then finding whether a subgroup is $\lambda$-essential becomes significantly less complicated.




\end{document}